\documentclass[12pt]{amsart}
\usepackage{amsmath}
\usepackage{amsxtra}
\usepackage{amstext}
\usepackage{amssymb}
\usepackage{amsthm}
\usepackage{latexsym}
\usepackage{dsfont} % for \mathbf{N}
\usepackage{verbatim}
\usepackage{tabls}
\usepackage{graphicx}
\usepackage{rotating}
\usepackage{caption}
\theoremstyle{plain}
\newtheorem{thm}{Theorem}[section]
\newtheorem{cor}[thm]{Corollary}
\newtheorem{lem}[thm]{Lemma}
\newtheorem{prop}[thm]{Proposition}

\theoremstyle{definition}
\newtheorem{defn}[thm]{Definition}

\numberwithin{equation}{section}
\newcommand{\ci}[1]{_{{}_{\scriptstyle{#1}}}}

\def\R1{\widetilde{R}}
\def\T1{\widetilde{T}}
\def\Tbar{\overline{T}}

\def\dist{\operatorname{dist}}
\def\supp{\operatorname{supp}}
\def\Lip{\operatorname{Lip}}
\def\eps{\varepsilon}
\def\kap{\varkappa}

\def\loc{\operatorname{loc}}
\def\e{\textbf{e}}
\def\R{\mathbb{R}}
\def\MBD{\mathfrak{M}_{\operatorname{bounded}}}
\def\MS{\mathfrak{M}_{\operatorname{smooth}}}
\def\MD{\mathfrak{M}_{\operatorname{decent}}}
\def\BMD{\overline{\mathfrak{M}_{\operatorname{decent}}}}

\def\pleas{\operatorname{diffuse}}

\def\XXint#1#2#3{{\setbox0=\hbox{$#1{#2#3}{\int}$}
     \vcenter{\hbox{$#2#3$}}\kern-.5\wd0}}

\begin{document}

\title[Reflectionless measures I]
{Reflectionless measures for Calder\'{o}n-Zygmund operators I:\\
General Theory.}

\author[B. Jaye]
{Benjamin Jaye}
\address{Department of Mathematical Sciences,
Kent State University,
Kent, OH 44240, USA}
\email{bjaye@kent.edu}

\author[F. Nazarov]
{Fedor Nazarov}
\email{nazarov@math.kent.edu}

\date{\today}

\begin{abstract}  We study the properties of reflectionless measures for an $s$-dimensional Calder\'{o}n-Zygmund operator $T$ acting in $\mathbb{R}^d$, where $s\in (0,d)$.  Roughly speaking, these are the measures $\mu$ for which $T(\mu)$ is constant on the support of the measure. In this series of papers, we develop the basic theory of reflectionless measures, and describe the relationship between the description of reflectionless measures and certain well-known problems in harmonic analysis and geometric measure theory.  \end{abstract}

\maketitle

\section{Introduction}

Fix an integer $d\geq 2$, and let $s\in (0,d)$.  For a measure $\mu$,  the $s$-dimensional Calder\'{o}n-Zygmund operator (CZO) $T_{\mu}$ is formally defined by $T_{\mu}(f)(x) = \int_{\mathbb{R}^d} K(x-y)f(y)\,d\mu(y)$, where $K$ is an odd $s$-dimensional Calder\'{o}n-Zygmund (CZ)-kernel (see Section \ref{primer} for the definition).

In this series of papers, we consider the structure and properties of reflectionless measures for $s$-dimensional CZOs.  These are the measures $\mu$ for which $T_{\mu}(1)$ is constant on the support of $\mu$, where this statement is interpreted in a suitable weak sense.  This study was motivated by certain problems concerning the geometry of measures with bounded $s$-Riesz transform, the CZO with kernel $K(x) = \tfrac{x}{|x|^{s+1}}$.

\subsection{Description of Contents.} Before we describe the general theory of reflectionless measures further, we pause to give a brief outline of the contents of the three papers which comprise this series.\\

\textbf{Part I}  is devoted to the general theory of reflectionless measures.   The notion of a reflectionless measure is introduced, and basic regularity properties of reflectionless measures are presented, along with some tools that will be needed in Parts II and III.\\

\textbf{Part II}  concerns the study of \textit{rigid} CZOs, i.e., the operators that have very few reflectionless measures associated to them.   We show that the measures $\mu$ for which a rigid CZO $T$ is bounded in $L^2(\mu)$ must have robust geometric structure.  Here the geometric structure is either given in terms of the uniform rectifiability of the support of $\mu$, or in terms of the boundedness of a \textit{positive} Wolff potential of power type
$$\mathbb{W}_p(\mu)(x) = \int_0^{\infty}\Bigl(\frac{\mu(B(x,r))}{r^s}\Bigl)^p \frac{dr}{r},
$$
with $p\in (0,\infty)$.

In \cite{JN1}, we gave a new proof of the Mattila-Melnikov-Verdera theorem on the uniform rectifiability of an Ahlfors-David regular measure with bounded Cauchy transform operator, which relied on a description of the Ahlfors-David regular reflectionless measures associated to the Cauchy transform.

In Part II, a new quantitative version of a result of Eiderman, Nazarov and Volberg is proved.  This result states that, if $s\in (d-1,d)$, the boundedness of the $s$-Riesz transform of a non-atomic finite measure $\mu$ yields that $\int_{\mathbb{R}^d}\mathbb{W}_p(\mu)(x)\,d\mu(x)<\infty$ for some large $p\in(0,\infty)$.  The integrability of a power-type Wolff potential with $p=2$ was conjectured by Mateu, Prat, and Verdera \cite{MPV}, which would be sharp (see \cite{ENV2} where the reverse inequality was proved).   Unfortunately, our methods do not currently yield the sharp value of $p$.\\

\textbf{Part III} studies \textit{relaxed} CZOs, i.e., the operators with an abundance of reflectionless measures.   For a relaxed CZO $T$, we show that there exists a measure $\mu$  supported on a `wild' set for which the associated CZO transform $T_{\mu}$ exhibits nice behaviour.  Here the wildness of a set can be measured in terms of the set being unrectifiabile, or, more generally, having an irregular density function.

\subsection{The General Theory}  The general theory begins with a thorough development of the weak notion of a reflectionless measure.  The development may initially appear a little fastidious, but a careful exposition appears necessary as our results are proved by studying $T$ from two standpoints:\smallskip

--  as an operator assigning a potential $T\nu$, defined almost everywhere in $\mathbb{R}^d$ with respect to the $d$-dimensional Lebesgue measure, for any finite signed measure $\nu$,\smallskip

and,\smallskip

--  as an operator acting from a certain linear space of measures to certain space of distributions associated with the measure $\mu$.  \\

Beginning in Section \ref{nicesection}, we discuss the regularity properties of reflectionless measures and their associated potentials.  Our main results  are proved as consequences of a certain technical result called the \textit{Collapse Lemma} (Proposition \ref{collem}), which provides a link between the density of a reflectionless measure $\mu$ in a ball centred on the support of $\mu$ and the size of a certain modified potential $\Tbar_{\mu}(1)$ of $\mu$ in the same ball.  We mention here two consequences.

Assume that $\mu$ is a reflectionless  measure satisfying the growth condition $\mu(B(x,r))\leq r^s$ for any every ball $B(x,r)$.  Let $\eps>0$. \\

\textbf{1.  (Non-degeneracy)}  There exist $M=M(\eps)>0$ and $\tau=\tau(\eps)>0$  such that if $|\Tbar_{\mu}(1)(x)|>\eps$, then $\mu(B(x, Mr))\geq \tau r^s$.\\

\textbf{2.  (Porosity)}  There is a constant $c(\eps)>0$ such that whenever $\tfrac{1}{m_d(B(x,r))}\int_{B(x,r)} |\Tbar_{\mu}(1)|dm_d >\eps$ for a ball $B(x,r)$, there is a ball $B'\subset B(x,r)$ of radius $c(\eps)r$ that is disjoint from the support of $\mu$ (here $m_d$ denotes the $d$-dimensional Lebesgue measure).\\

In \cite{JN2}, we showed that the two dimensional Lebesgue measure restricted to a disk is a reflectionless measure for the CZO with the kernel $K(z)=\tfrac{\overline{z}}{z^2}$.  This example shows that one cannot expect a property like porosity to hold without an additional size assumption on the associated potential, at least in the generality in which we are working here.  The existence of this reflectionless measure was used in \cite{JN2} to construct a purely unrectifiable measure $\nu$ with respect to which the CZO with kernel $\tfrac{\overline{z}}{z^2}$ is a bounded operator in $L^2(\nu)$.  The general scheme behind constructions of this type will be presented in Part III.

The non-degeneracy property will form one of the main points behind our proof of the quantitative version of the Eiderman-Nazarov-Volberg theorem  in Part II.\\

We conclude the introduction by making several remarks about previous literature where objects similar to (what we call) reflectionless measures have been considered.  In spectral theory, a Jacobi matrix is called reflectionless on a set $E\subset \mathbb{R}$ is its Green function has zero imaginary part on $E$.  These matrices have proved to be useful in the study of the absolutely continuous spectrum of one dimensional discrete Schr\"{o}dinger operators, see Remling \cite{Rem} and Poltoratski-Remling \cite{PR}.  The description of finite measures $\nu$ on $\mathbb{C}$ for which the corresponding Cauchy transform $C(\nu)$ (understood in the principal value sense) vanishes $\nu$-almost everywhere has also attracted the attention of many authors.  The full characterization of them is still unknown, though partial progress has been made by Tolsa and Verdera \cite{TV}, and Melnikov, Poltoratski, and Volberg \cite{MPV}.

Finally, it would be remiss if we did not mention Mattila's notion of a symmetric measure.  These are measures $\nu$ for which $$\int_{B(x,r)}|x-y|^sK(x-y)d\nu(y)=0$$ for $\nu$-almost every $x\in \supp(\nu)$ and every $r>0$.  Symmetric measures naturally arise as a useful tool in the study of measures $\mu$ for which the potential $T(\mu)$ exists in the principal value sense $\mu$-almost everywhere.  This existence of principle values is stronger\footnote{Stronger in the sense that if  $\nu$ is finite and the CZO potential $T\nu$ exists in the principal value sense (or even if the maximal CZO potential is point-wise bounded $\nu$-almost everywhere), then for each $\eps>0$, there is a set $E$ with $\nu(E)\geq (1-\eps)\nu(\mathbb{R}^d)$, such that the CZO $T_{\nu'}$ is bounded in $L^2(\nu')$ with $\nu' = \chi_E\nu$.} than just the $L^2(\mu)$ boundedness of $T$ and the geometric properties of symmetric measures are understood better (see \cite{Mat2, MP}).  However, there are still many open problems regarding their structure, see Chapter 14 of \cite{Mat}.

\section{Notation}  Fix an integer $d\geq  2$ and a real number $s\in (0,d)$.

By a measure, we shall always mean a non-negative locally finite Borel measure.  We shall also make use of (real valued) signed measures, but these shall always be explicitly identified as such. For a (signed) measure $\mu$, $\supp(\mu)$ denotes its closed support.  The $d$-dimensional Lebesgue measure is denoted by $m_d$.

For a set $E$, and $\delta>0$, $E_{\delta}$ denotes the open $\delta$ neighbourhood of $E$.

Fix another integer $d'\geq 1$.  The integral kernels in this paper are $\mathbb{C}^{d'}$ valued.  This will be important in applications of this theory, but it causes a little bit of notational hassle that we now address.

For two scalar (complex) valued functions $f,g\in L^2(\mu)$, we define
$$\langle f,g\rangle_{\mu} = \int_{\mathbb{R}^d}  f  g \,\,d\mu
$$
(the reader should not be worried that there is no complex conjugation sign upon $g$).  In the event that one of the two functions (say $f$) is $\mathbb{C}^{d'}$ valued, we shall write $\langle f,g\rangle_{\mu}$ to mean the vector with components $\langle f_j,g\rangle_{\mu} $, where $f_j$ are the components of $f$.

For a function $f$ defined everywhere on $\R^d$, we define
$$\|f\|_{\infty} = \sup_{x\in \R^d}|f(x)|.$$

A function $f$ (either scalar or vector valued) is called Lipschitz continuous if
$$\|f\|_{\Lip} = \sup_{x,y\in \mathbb{R}^d,\; x\neq y} \frac{|f(x)-f(y)|}{|x-y|}<\infty.
$$
For $\alpha\in (0,1]$, $f$ is said to be $\alpha$-H\"{o}lder continuous if
$$\|f\|_{\Lip^{\alpha}} = \sup_{x,y\in \mathbb{R}^d,\; x\neq y} \frac{|f(x)-f(y)|}{|x-y|^{\alpha}}<\infty.
$$
In particular, $\|f\|_{\Lip^{1}} = \|f\|_{\Lip}$.

For a Borel set $U\subset \mathbb{R}^d$, we shall make use of a few spaces of functions:

-- $\Lip^{\alpha}(U)$ denotes the set of $\alpha$-H\"{o}lder continuous functions on $U$.

-- $\Lip_0(U)$ denotes the set of Lipschitz continuous functions that are \textit{compactly supported in the interior} of $U$.

-- The set $\Lip_{B}(U)$ denotes the set of \textit{bounded} Lipschitz functions on $U$.\\

Normally, we shall denote a large positive constant by $C$ and a small positive constant by $c$.  When new constants have to be defined in terms of some previously chosen ones (like in delicate iteration arguments in the second half of the paper), we number them.  The conventions are that all constants may depend on $d$, $s$, $\alpha$, $\|K\|_{*}$ and $\Lambda$ in addition to parameters explicitly mentioned in parentheses\footnote{Here $\alpha$ and $\|K\|_{*}$ will be parameters governing the regularity of a CZ kernel $K$, while $\Lambda$ shall denote a parameter governing a regularity property of the measure under consideration} and a numbered constant with index $j$ can be chosen in terms of constants with  indices less than $j$ (say, $C_{12}$ can be chosen in terms of $c_4$ and $C_{10}$).

\section{Calder\'{o}n-Zygmund operators of dimension $s$ and associated bilinear forms} \label{primer}

We study the properties of $T$ from two standpoints:  as properties of an operator assigning a potential, locally integrable with respect to $m_d$, to every signed finite Borel measure in $\mathbb{R}^d$, and as properties of an operator acting from some linear space of measures to an appropriate space of generalized functions associated with the measure $\mu$ via an extension of the standard duality $\langle \,\cdot\,,\, \cdot\,\rangle_{\mu}$.

For the purposes of this section, we shall call $K:\mathbb{R}^d\backslash \{0\} \to \mathbb{C}^{d'}$ an $s$-dimensional Calder\'{o}n-Zygmund (CZ) kernel if the following properties are satisfied\smallskip

(i). $|K(x)|\leq \tfrac{1}{|x|^s}$ for all $x\in \mathbb{R}^d\backslash \{0\}$.
\smallskip

(ii).  $K(-x) = -K(x)$ for all $x\in \mathbb{R}^d\backslash \{0\}$.\smallskip

(iii).  For some $\alpha\in (0,1]$, the $\alpha$-H\"{o}lder norm $\|K\|_{*}$ of the function $x \mapsto K(x)|x|^{s+\alpha}$ (extended so that it equals $0\in \mathbb{C}^{d'}$ at $0\in \mathbb{R}^d$) is finite.\smallskip

 It is straightforward to see that the third condition implies that, for every $x,x'\in \mathbb{R}^d\backslash\{0\}$ with $|x'|\leq |x|$, one has
\begin{equation}\label{standard}|K(x)-K(x')|\leq C\frac{|x-x'|^{\alpha}}{|x'|^{s+\alpha}},
\end{equation}
where the constant $C$ depends on $s$, $\alpha$ and $\|K\|_{*}$ only.

In later sections of this paper, it will be convenient to impose an additional requirement of homogeneity upon the kernel, but imposing this condition at this point is a nuisance.

Fix a CZ-kernel $K$.  We start with a useful inequality that will allow us to establish the basic properties of the potential $T\nu$ as an $L^1_{\text{loc}}(m_d)$-function.

\begin{lem}\label{locl1} There is a constant $C>0$, depending on $d$ and $s$, such that for any measure $\nu$,
$$\int_{B(x,r)}\int_{B(y,R)}\frac{1}{|z-y|^s} d\nu(z)dm_d(y) \leq C \min(r,R)^{d-s}\nu(B(x,r+R)),
$$
for any $x\in \mathbb{R}^d$, and $r,R>0$.
\end{lem}
To prove the lemma, first note that by applying the Fubini-Tonelli theorem, the left hand side of the stated inequality equals $$\int_{B(x,r+R)}\int_{B(x,r)\cap B(z,R)}\frac{1}{|z-y|^s} dm_d(y)d\nu(z).$$  But for any $x,z\in \mathbb{R}^d$, $\int_{B(x,r)\cap B(z,R)}\frac{1}{|z-y|^s} dm_d(y)\leq C\min(r,R)^{d-s}$, and the lemma follows.

This lemma immediately implies that if $\nu$ is a finite signed measure, then $T\nu(\,\cdot\,)=\int_{\mathbb{R}^d}K(\cdot - y)d\nu(y)$ is defined $m_d$-almost everywhere as a Lebesgue integral and is locally integrable with respect to $m_d$.\\

We will now define $T\nu$ for a signed measure $\nu$ as a generalized function acting on test functions $\varphi\in \Lip_0(\mathbb{R}^d)$.  Fix a non-negative measure $\mu$.  Our goal is to make sense of the bilinear form $\langle T\nu, \varphi\rangle_{\mu}$ so that it coincides with $\int_{\R^d\times\R^d}K(x-y)\varphi(x)d\nu(y)\,d\mu(x)$ whenever the latter makes sense as a Lebesgue integral.   It may not be possible to do this for all finite signed Borel measures $\nu$ in general, so we shall restrict ourselves to some linear space of `decent' signed measures (where the exact meaning of the word `decent' will depend on the measure $\mu$).

In most arguments below, it will be convenient to have the test measures $\nu$ compactly supported, so we will include this condition into our definition of `decency' even if it does not seem immediately necessary.

\subsection{The space $\MBD(\mu)$}  With $\varphi\in \Lip_0(\mathbb{R}^d)$, we have no problem with the integral $\int_{\R^d\times\R^d}K(x-y)\varphi(x)d\nu(y)\,d\mu(x)$ if $\int_{\R}\tfrac{1}{|x-y|^s}d|\nu|(y)\in L^1_{\loc}(\mu)$, so we can declare any measure with this property `decent'.  Lemma \ref{locl1} shows that the linear space $\MBD(\mu)$ of such compactly supported signed measures is rich.  In particular, if $\mu$ is locally finite (which we always assume throughout the paper), then the Dirac mass $\delta_{x}\in \MBD(\mu)$ for $m_d$-almost every $x\in \R^d$.

\subsection{The space $\MS(\mu)$}  On the other hand, if $\nu=f\mu$  and if the map $(x,y)\mapsto K(x-y)f(y)\varphi(x)$  is Lebesgue integrable with resect to $\mu\times\mu$, then we can use the anti-symmetry of the kernel $K$ to write
\begin{equation}\begin{split}\nonumber\langle T\nu, \varphi\rangle_{\mu} & = \iint_{\R^d\times\R^d}\!\!\!K(x-y)f(y)\varphi(x)\,d\mu(x)\,d\mu(y)\\
&= \iint_{\R^d\times\R^d}\!\!\!K(x-y)H_{f,\varphi}(x,y)\,d\mu(x)\,d\mu(y),
\end{split}\end{equation}
where $H_{f,\varphi}(x,y) = \frac{1}{2}\bigl[f(y)\varphi(x) - f(x)\varphi(y)\bigl]$.

If $f,\varphi\in \Lip_0(\mathbb{R}^d)$, then $H_{f,\varphi}\in \Lip_0(\R^d\times \R^d)$ and $H_{f,\varphi}(x,x)=0$ for $x\in \mathbb{R}^d$.  Thus $|H_{f,\varphi}(x,y)|\leq C_{f,\varphi}|x-y|$.  This decay near the diagonal can potentially cancel the singularity in the kernel $K$ so that the function $(x,y)\mapsto K(x-y)H_{f,\varphi}(x,y)$ may belong to $L^1(\mu\times\mu)$ even if $K(x,y)f(y)\varphi(x)$ does not.  For instance, this happens if $(x,y)\mapsto \tfrac{1}{|x-y|^{s-1}}$ is locally integrable with respect to $\mu\times\mu$.

\begin{defn}\label{pleasdef} A measure $\mu$ is \textit{diffuse} if  the function $(x,y)\mapsto \tfrac{1}{|x-y|^{s-1}}$ is locally integrable with respect to $\mu\times\mu$.  \end{defn}

If $\mu$ is $\pleas$, it is natural to define the linear space of `smooth' signed measures $\nu$ by  $\MS(\mu) = \bigl\{ f\mu : f\in \Lip_0(\R^d)\bigl\}.$ We now put
$$\MD(\mu) = \MBD(\mu)+\MS(\mu),
$$
with the bilinear form $\langle T\nu, \varphi\rangle_{\mu}$ defined as
\begin{equation}\begin{split}\nonumber\iint_{\R^d\times\R^d} &K(x-y)\varphi(x)d\nu_{\operatorname{bdd}}(y)\,d\mu(x) \\
&+ \iint_{\R^d\times\R^d}K(x-y)H_{f,\varphi}(x,y)\,d\mu(x)\,d\mu(y),
\end{split}\end{equation}
for $\nu = \nu_{\operatorname{bdd}}+f\mu\in \MD(\mu)$.\smallskip

The discussion above shows that the definition is consistent in the sense that it does not depend on the representation of $\nu$ as a sum of its bounded ($\nu_{\operatorname{bdd}}$) and smooth  ($f\mu$) parts, and, moreover, we can use the non-symmetrised version of the formula instead of the symmetrised one whenever the corresponding Lebesgue integral makes sense.  Let us explicitly remark here that any signed measure $\nu\in \MD(\mu)$ is compactly supported, and that the space $\MD(\mu)$ does not depend on the choice of a CZ-kernel.

\subsection{Balanced measures $\BMD(\mu)$} We call a finite signed measure $\nu$ balanced if $\nu(\mathbb{R}^d)=0$.  By $\BMD(\mu)$, we denote the linear space of balanced measures in $\MD(\mu)$.

Note that if $\nu \in \BMD(\mu)$, then for sufficiently large $x\in \R^d$ (namely for all points $x$ for which $\supp(\nu)\subset B(0, \tfrac{|x|}{2})$\,), the value $T\nu(x) = \int_{\R^d}K(x-y)\,d\nu(y)$ is well defined and we can write
$$|T\nu(x)| =\Bigl|\int_{\R^d} \bigl[K(x-y) - K(x)\bigl]d\nu(y)\Bigl|\leq \frac{C}{|x|^{s+\alpha}}\int_{\mathbb{R}^d}|y|^{\alpha}d|\nu|(y).
$$
Thus, if the measure $\mu$ has \textit{restricted growth at infinity} in the sense that $\int_{|x|\geq 1} \tfrac{1}{|x|^{s+\alpha}}\,d\mu(x)<\infty$ (we shall always make this assumption from now on), we can couple $T\nu$ with \textit{bounded Lipschitz functions} $\varphi\in \Lip_B(\mathbb{R}^d)$ rather than just compactly supported ones by defining the corresponding bilinear form $\langle T\nu, \varphi\rangle_{\mu}$ as the limit $$\langle T\nu,\varphi\rangle_{\mu}=\lim_{k\to \infty}\langle T\nu, \psi_k\varphi\rangle_{\mu}$$ for any sequence $\psi_k$ of $\Lip_0(\R^d)$ functions satisfying $0\leq \psi_k\leq 1$ on $\R^d$ and $\psi_k \equiv 1$ on $B(0,k)$, say.  Alternatively, we can write this limit as
$$\langle T\nu,\varphi\rangle_{\mu}=\langle T\nu, \psi\varphi\rangle_{\mu}+\int_{\mathbb{R}^d} T\nu(x)[1-\psi(x)]\varphi(x)\,d\mu(x),
$$
where $\psi\in\Lip_0(\mathbb{R}^d)$ is identically $1$ on some neighbourhood of $\supp(\nu)$.

\subsection{Compactly supported kernels}\label{cmptsupp}  There is another obvious case when $\langle T\nu, \varphi\rangle_{\mu}$ makes sense for every $\nu\in \MD(\mu)$ and $\varphi\in \Lip_B(\mathbb{R}^d)$, namely, the case when the kernel $K$ is compactly supported.  It is easy to check that in this case the sequence $\langle T\nu, \psi_k\varphi\rangle_{\mu}$ stabilizes eventually, and $\langle T\nu, \varphi \rangle_{\mu}$ can be defined as $\langle T\nu, \psi\varphi \rangle_{\mu}$ for any $\Lip_0(\mathbb{R}^d)$-function $\psi$ that is identically $1$ on a sufficiently large ball centred at the origin.  In particular, if $\nu=f\mu$, $f\in \Lip_0(\R^d)$, then for $\varphi\in \Lip_{B}(\R^d)$,
$$\langle T\nu, \varphi\rangle_{\mu} = \iint_{\R^d\times\R^d}K(x-y)H_{f,\varphi}(x,y)d\mu(x)d\mu(y).  $$

\subsection{The distribution $\T1\nu$}  Even if the measure $\nu \in \MD(\mu)$ is not balanced, we can balance it by subtracting its total mass times some fixed probability measure $\nu_0\in \MD(\mu)$.  For $\nu_0$, we can take any Dirac point mass $\delta_x\in \MBD(\mu)$, like it was done in \cite{JN1}.  However, for the purposes of the present exposition, it will be more convenient to choose $\nu_0 = \eta \mu\in \MS(\mu)$, with some $\eta\in \Lip_0(\mathbb{R}^d)$ satisfying $\int_{\mathbb{R}^d}\eta \,d\mu=1$ (such a function can be found unless $\mu\equiv 0$, in which case all our constructions trivialize to rewriting the identity $0=0$ in various forms).

So, we just define the bilinear form
$$\langle \T1\nu, \varphi\rangle_{\mu} =\langle T[\nu-\nu(\mathbb{R}^d)\nu_0], \varphi \rangle_{\mu}
$$
(here $\nu\in \MD(\mu)$ and $\varphi\in \Lip_B(\mathbb{R}^d)$).

Note that in the case of a compactly supported kernel $K$, we can write
\begin{equation}\label{compactsplit}\langle \T1\nu, \varphi\rangle_{\mu}=\langle T\nu, \varphi\rangle_{\mu} - \nu(\mathbb{R}^d)\langle T\nu_0, \varphi\rangle_{\mu}.\end{equation}

\subsection{The regularized kernel $K_{\delta}$ and the localized kernel $K^{\delta}$}

For $x\in \mathbb{R}^d$, $x\neq 0$, define the  kernels $$K_{\delta}(x) = K(x)\Bigl(\frac{|x|}{\max(\delta, |x|)}\Bigl)^{s+\alpha}, \text{ and }K^{\delta}(x) = K(x)-K_{\delta}(x).$$
The domain of the kernel $K_{\delta}$ is then extended to the entire space by setting its value at the origin to be $0$.

We first claim that $K_{\delta}$ (and, hence, $K^{\delta}$) is an $s$-dimensional CZ-kernel\footnote{It is for this reason why homogeneity is not made part of the definition of a CZ kernel in this section.} and, moreover, its Calder\'{o}n-Zygmund bounds are controlled by a constant independent of $\delta$.

This is clear for the growth bound $|K_{\delta}(x)|\leq \tfrac{1}{|x|^s}$.   To estimate the H\"{o}lder norm of $x\mapsto |x|^{s+\alpha}K_{\delta}(x)$, we shall appeal to a simple lemma.

\begin{lem}Suppose that $f\in \Lip^{\alpha}(\R^d)$, $f(0)=0$, and $\|f\|_{\Lip^{\alpha}}\leq 1$, and for some $\delta>0$, $g_{\delta}\in \Lip^{\alpha}(\R^d)$, $\supp(g_{\delta})\subset \overline{B(0,\delta)}$, and $\|g_{\delta}\|_{\Lip^{\alpha}}\leq \tfrac{A}{\delta^{\alpha}}$ for some $A>0$.  Then $h_{\delta}=fg_{\delta}\in \Lip^{\alpha}(\R^d)$, and moreover $\|h_{\delta}\|_{\Lip^{\alpha}}\leq 2A$.\end{lem}

\begin{proof}Fix $x,y\in \R^d$.  The estimation of $|h_{\delta}(x)-h_{\delta}(y)|$ is trivial unless either $x$ or $y$ in $B(0,\delta)$.  If both $x,y\in \overline{B(0,\delta)}$, then we use the triangle inequality to write
$$|h_{\delta}(x)-h_{\delta}(y)|\leq |g_{\delta}(x)||f(x)-f(y)|+|f(y)||g_{\delta}(x)-g_{\delta}(y)|.
$$
Choosing $z\not\in B(0,\delta)$ with $|z-y|\leq \delta$, we see that $|g_{\delta}(y)|=|g_{\delta}(y)-g_{\delta}(z)|\leq \tfrac{A}{\delta^{\alpha}}|y-z|^{\alpha}\leq A$, while $|f(y)|=|f(y)-f(0)|\leq |y|^{\alpha}\leq \delta^{\alpha}$.  Thus $|h_{\delta}(x)-h_{\delta}(y)|\leq 2A|x-y|^{\alpha}$.  It remains to consider the case where one of the points, say $x$, lies in $B(0,\delta)$, while $y\not\in B(0,\delta)$.  Denote by $y^*$ the point on the line segment $[x,y]$ with $|y^*|=\delta$.  Then $|x-y^*|\leq |x-y|$.  Since $g_{\delta}$ is supported in $\overline{B(0,\delta)}$, $|h_{\delta}(x)-h_{\delta}(y)| =|h_{\delta}(x)-h_{\delta}(y^*)|$.  But since $x,y^*\in \overline{B(0,\delta)}$, the previously considered case yields that $|h_{\delta}(x)-h_{\delta}(y^*)|\leq 2A|x-y^*|^{\alpha}\leq 2A|x-y|^{\alpha}$.  The lemma is proved.\end{proof}

We shall apply this lemma with the functions $f(x)=\tfrac{K(x)|x|^{s+\alpha}}{\|K\|_{*}}$, and $g_{\delta}(x) = 1-\bigl(\frac{|x|}{\max(\delta, |x|)}\bigl)^{s+\alpha}$.  Then $\|f\|_{\Lip^{\alpha}}\leq 1$, and so the triangle inequality yields that $\|f(1-g_{\delta})\|_{\Lip^{\alpha}(\R^d)}\leq 1+ \|fg_{\delta}\|_{\Lip^{\alpha}(\R^d)}$.  Thus, in order to prove that the H\"{o}lder norm of the function $x\mapsto |x|^{s+\alpha}K_{\delta}(x)[=\|K\|_{*}f(x)(1-g_{\delta}(x))]$ can be estimated independently of $\delta$, it suffices to show that there is some $A>0$ (independent of $\delta$) so that $\|g_{\delta}\|_{\Lip^{\alpha}}\leq\tfrac{A}{\delta^{\alpha}}$.    But note that $g_{\delta}(\cdot) = g(\tfrac{\cdot}{\delta})$, where $g(x)= 1-\bigl(\frac{|\cdot|}{\max(1, |\cdot|)}\bigl)^{s+\alpha} = \min (1,|x|)^{s+\alpha}$, so we only need to check that $g\in \Lip^{\alpha}(\R^d)$.

To confirm the required H\"{o}lder continuity of $g$,  first note that the function $x\mapsto |x|^{\alpha}$ is $\alpha$-H\"{o}lder continuous on $\R^d$ (with $\alpha$-H\"{o}lder norm equal to $1$).   Second, the function $t\mapsto \min(1,t)$ is Lipschitz continuous on $\R$ (with Lipschitz norm equal to $1$).  Consequently, $x\mapsto \min(1,|x|^{\alpha})\in \Lip^{\alpha}(\R^d)$ (with $\alpha$-H\"{o}lder norm bounded by $1$), and has its values in the interval $[0,1]$.   Finally, the function $t\mapsto t^{(s+\alpha)/\alpha}$ lies in $\Lip([0,1])$ (with Lipschitz norm on $[0,1]$ at most $\tfrac{s+\alpha}{\alpha}$).  Thus, $g$ lies in $\Lip^{\alpha}(\R^d)$, and its $\alpha$-H\"{o}lder norm is no greater than $\tfrac{s+\alpha}{\alpha}$.

\subsection{The regular operator $T_{\delta}$ and the local operator $T^{\delta}$}  We can now apply all the above constructions to the kernels $K_{\delta}$ and $K^{\delta}$ instead of $K$, and define the corresponding operators $T_{\delta}$, $T^{\delta}$, $\T1_{\delta}$ and $\T1^{\delta}$.  Since $K=K_{\delta}+K^{\delta}$, we have the identities
\begin{equation}\label{lineardecent}T=T_{\delta}+T^{\delta}\text{ and }\T1 = \T1_{\delta}+\T1^{\delta}\text{ on }\MD(\mu).\end{equation}

The operator $T^{\delta}$ is $\delta$-localized in the sense that the corresponding kernel $K^{\delta}$ is supported on a small ball (the closed ball of radius $\delta$ centred at the origin), so the bilinear form $\langle T^{\delta}(\nu), \varphi\rangle_{\mu}$ makes sense for every $\nu\in \MD(\mu)$, and every Lipschitz function $\varphi$ (even the boundedness of $\varphi$ is not necessary).   Moreover, the corresponding value of the bilinear form depends only on the values of $\varphi$ in the $\delta$-neighbourhood of $\supp(\nu)$. In particular,
\begin{equation}\begin{split}\label{localintest}\int_{B(x,r)}\!\!|\langle T^{\delta}[\delta_y], \varphi\rangle_{\mu}|\, dm_d(y)&\leq \int_{B(x,r)}\int_{B(y,\delta)}\!\frac{|\varphi(z)|}{|y-z|^s}\,d\mu(z)dm_d(y)\\
&\leq C\min(r,\delta)^{d-s}\mu(B(x,r+\delta)),
\end{split}\end{equation}
where Lemma \ref{locl1} has been used in the second inequality.

The operator $T_{\delta}$, on the other hand, has  bounded continuous kernel, so $T_{\delta}\nu = \int_{\mathbb{R}^d} K_{\delta}(\cdot-y)d\nu(y)$ is defined as a continuous function in $\R^d$ for \textit{every} finite signed Borel measure $\nu$, not only for $\nu\in \MD$.   Moreover, when $\varphi\in \Lip_0(\R^d)$,  the integral
$$\iint_{\R^d\times\R^d}K_{\delta}(x-y)\varphi(x)d\nu(y)\,d\mu(x)$$ makes sense as a usual Lebesgue integral, and, thereby, represents the bilinear form $\langle T_{\delta}\nu, \varphi\rangle_{\mu}$.  Applying the Fubini theorem, we conclude that this bilinear form can also be written as $\int_{\R^d}G_{\varphi,\delta}(y)d\nu(y)$, where $G_{\varphi,\delta}(y) = \int_{\R^d}K_{\delta}(x-y)\varphi(x)\,d\mu(x)=\langle T_{\delta}\delta_y, \varphi\rangle_{\mu}$ is a bounded continuous function in $\R^d$.  Now for $\varphi\in \Lip_0(\R^d)$, we have
\begin{equation}\begin{split}\nonumber\langle \T1_{\delta}\nu, \varphi\rangle_{\mu} & = \langle T_{\delta}(\nu-\nu(\R^d)\nu_0), \varphi\rangle_{\mu} = \int_{\R^d}G_{\varphi,\delta}(y)d[\nu-\nu(\R^d)\nu_0](y)\\
& = \iint_{\R^d\times\R^d}\bigl[G_{\varphi,\delta}(y) - G_{\varphi,\delta}(y')\bigl]d\nu_0(y')d\nu(y).
\end{split}\end{equation}
If  $\varphi\in \Lip_B(\R^d)$ is merely bounded, we can still write
\begin{equation}\begin{split}\nonumber \langle \T1_{\delta}\nu, \varphi\rangle_{\mu}&=\lim_{k\to \infty}\langle \T1_{\delta}\nu, \varphi\psi_k\rangle_{\mu}\\
&=\lim_{k\to \infty}\iint_{\R^d\times \R^d}\bigl[G_{\varphi\psi_k,\delta}(y) - G_{\varphi\psi_k,\delta}(y')\bigl]d\nu_0(y')d\nu(y).
\end{split}\end{equation}
Now note that despite the fact that $G_{\varphi\psi_k,\delta}(y)$ and $G_{\varphi\psi_k,\delta}(y')$ may fail to tend to a limit as $k\to \infty$, their difference
$$G_{\varphi\psi_k,\delta}(y) - G_{\varphi\psi_k,\delta}(y')=\int_{\R^d}\bigl[K_{\delta}(x-y)-K_{\delta}(x-y')\bigl]\varphi(x)\psi_k(x)\,d\mu(x)
$$
always tends to $\int_{\R^d}\bigl[K_{\delta}(x-y)-K_{\delta}(x-y')\bigl]\varphi(x)\,d\mu(x)$ uniformly on compact subsets of $\R^d\times\mathbb{R}^d$ due to the estimate
\begin{equation}\label{Kdeltest}|K_{\delta}(x-y)-K_{\delta}(x-y')|\leq C\min\Bigl(\frac{1}{\delta^s}, \frac{|y-y'|^{\alpha}}{\min(|x-y|, |x-y'|)^{s+\alpha}}\Bigl),
\end{equation}
and the restricted growth assumption $\int_{|x|\geq 1}\tfrac{1}{|x|^{s+\alpha}}\,d\mu(x)<\infty$.

Thus, for every $\varphi\in \Lip_B(\R^d)$,
\begin{equation}\label{deltamuint}\langle \T1_{\delta}\nu, \varphi\rangle_{\mu} = \int_{\R^d}\widetilde{G}_{\varphi,\delta}(y) d\nu(y),\end{equation} where \begin{equation}\begin{split}\label{tildgdef}\widetilde{G}_{\varphi,\delta}(y) &= \iint_{\R^d\times\R^d}\bigl[K_{\delta}(x-y)-K_{\delta}(x-y')\bigl]\varphi(x)\,d\mu(x)d\nu_0(y') .\end{split}\end{equation}
Note that $\widetilde{G}_{\varphi,\delta}(y)$ is a continuous function.

We shall make use of the following identiy for the function $\widetilde{G}_{\varphi,\delta}$. For two points $y,z\in \mathbb{R}^d$, and $0<\delta\leq \Delta$,
\begin{equation}\begin{split}\label{deltaDelta}
\widetilde{G}_{\varphi,\delta}(y) - \widetilde{G}_{\varphi,\Delta}(z) = &\int_{\mathbb{R}^d}[K_{\delta}(x-y)-K_{\Delta}(x-z)]\varphi(x)d\mu(x)\\&+ \langle T^{\delta}\nu_0, \varphi\rangle_{\mu}- \langle T^{\Delta}\nu_0, \varphi\rangle_{\mu}.
\end{split}\end{equation}

To prove this identity, fix $y,z\in \mathbb{R}^d$ and $0<\delta\leq \Delta$.  Since each of the following integrals converges absolutely, the difference
\begin{equation}\begin{split}\nonumber\iint_{\R^d\times\R^d}\bigl[&K_{\delta}(x-y)-K_{\delta}(x-y')\bigl]\varphi(x)\,d\mu(x)d\nu_0(y')\\&-\iint_{\R^d\times\R^d}\bigl[K_{\Delta}(x-z)-K_{\Delta}(x-y')\bigl]\varphi(x)\,d\mu(x)d\nu_0(y')
\end{split}\end{equation}
equals
\begin{equation}\begin{split}\label{anotherdiff}&\iint_{\mathbb{R}^d\times\R^d}[K_{\delta}(x-y)-K_{\Delta}(x-z)]\varphi(x)d\mu(x)d\nu_0(y') \\
&\;\;- \iint_{\mathbb{R}^d\times \R^d}[K_{\delta}(x-y')-K_{\Delta}(x-y')]\varphi(x)d\mu(x)d\nu_0(y').
\end{split}\end{equation}
But $\nu_0$ is a probability measure, so the first of the two integrals appearing in (\ref{anotherdiff}) equals $\int_{\mathbb{R}^d}[K_{\delta}(x-y)-K_{\Delta}(x-z)]\varphi(x)d\mu(x)$.  

Let us now examine the second integral appearing in (\ref{anotherdiff}):
$$\iint_{\mathbb{R}^d\times \R^d}[K_{\delta}(x-y')-K_{\Delta}(x-y')]\varphi(x)d\mu(x)\eta(y')d\mu(y').
$$
Notice that the CZ-kernel $K_{\delta}-K^{\Delta}$ is compactly supported, and so we may use antisymmetry to rewrite this integral as
\begin{equation}\label{someintegral}\iint_{\mathbb{R}^d\times \R^d}[K_{\delta}(x-y')-K_{\Delta}(x-y')]H_{\eta,\varphi}(x,y')d\mu(x)d\mu(y'),
\end{equation}
(even though $\varphi\in \Lip_B(\R^d)$, see Section \ref{cmptsupp}).  Since $\mu$ is diffuse, the set $\{(x,y')\in \R^d\times\R^d: x=y'\}$ is a set of $\mu\times\mu$-measure zero.  Thus $K_{\delta}(x-y')-K_{\Delta}(x-y')=K^{\Delta}(x-y')-K^{\delta}(x-y')$ for $\mu\times\mu$ almost every $(x,y')$.  Consequently, the integral (\ref{someintegral}) equals 
$\langle T^{\Delta}(\eta\mu),\varphi\rangle_{\mu}- \langle T^{\delta}(\eta\mu),\varphi\rangle_{\mu}.$
Recalling again that $\nu_0=\eta\mu$, we find that the identity (\ref{deltaDelta}) has been proved.

\subsection{The functions $\Tbar_{\mu, \delta}(\varphi)$ and $\Tbar_{\mu}(1)$}

\begin{defn}For a bounded Lipschitz function $\varphi$, define \begin{equation}\label{Tbardeltdef}\Tbar_{\mu, \delta}(\varphi) = \widetilde{G}_{\varphi, \delta} -\langle T^{\delta}\nu_0, \varphi\rangle_{\mu}.\end{equation}
\end{defn}

Note that $\int_{\R^d} \Tbar_{\mu, \delta}(\varphi) d\nu \stackrel{(\ref{deltamuint})}{=} \langle \T1_{\delta} \nu, \varphi\rangle_{\mu} - \nu(\R^d)\langle T^{\delta}\nu_0, \varphi\rangle_{\mu}$, which, for $\nu\in \MD(\mu)$, can also be written as
\begin{equation}\label{bardecentrep}\begin{split}\int_{\R^d} \Tbar_{\mu, \delta}(\varphi) d\nu &\stackrel{(\ref{lineardecent})}{=} \langle \T1\nu, \varphi\rangle_{\mu}-\langle \T1^{\delta}\nu, \varphi\rangle_{\mu} -\nu(\mathbb{R}^d)\langle T^{\delta}(\nu_0),\varphi\rangle_{\mu}\\
& \stackrel{(\ref{compactsplit})}{=} \langle \T1\nu, \varphi\rangle_{\mu} - \langle T^{\delta}\nu, \varphi\rangle_{\mu}.
\end{split}\end{equation}

\begin{defn}For $\varphi\in \Lip_B(\R^d)$, and $x$ so that $\delta_x\in \MBD(\mu)$, define
\begin{equation}\label{prodrep}\Tbar_{\mu}(\varphi)(x) = \langle \T1\delta_x,\varphi\rangle_{\mu}.
\end{equation}\end{defn}
\vspace{0.1in}
While $\Tbar_{\mu}(\varphi)(x)$ is not defined for every $x\in \mathbb{R}^d$, it is a well defined Lebesgue measurable function with respect to $m_d$.   Notice that whenever $\delta_x\in \MBD(\mu)$, and $\delta>0$, we may apply the formula (\ref{bardecentrep}) with $\nu = \delta_x$ to yield
\begin{equation}\label{deltprodrep}\Tbar_{\mu,\delta}(\varphi)(x) = \langle \T1\delta_x,\varphi\rangle_{\mu} - \langle T^{\delta}[\delta_x],\varphi\rangle_{\mu},\end{equation}
and so
\begin{equation}\label{deltprodrep2}\Tbar_{\mu}(\varphi)(x) = \Tbar_{\mu, \delta}(\varphi)(x) +\langle T^{\delta}[\delta_x],\varphi\rangle_{\mu}.\end{equation}

\subsection{Useful Identities} We conclude this section of the paper by collecting some useful identities for the functions $\Tbar_{\mu,\delta}(1)$ and $\Tbar_{\mu}(1)$.

$\bullet$ For any $x,x'\in \mathbb{R}^d$, there is a useful identity for the difference $\Tbar_{\mu, \delta}(\varphi)(x)-\Tbar_{\mu, \delta}(\varphi)(x')$,
\begin{equation}\begin{split}\label{deltdiffform}\Tbar_{\mu, \delta}(\varphi)(x)-&\Tbar_{\mu, \delta}(\varphi)(x')=\widetilde{G}_{\varphi, \delta}(x)-\widetilde{G}_{\varphi, \delta}(x') \\
&\stackrel{(\ref{deltaDelta})}{=}  \int_{\mathbb{R}^d}[K_{\delta}(y-x)-K_{\delta}(y-x')]\varphi(y)\,d\mu(y).
\end{split}\end{equation}\smallskip

$\bullet$ Our next observation is that if $x,x'$ are points for which $\delta_{x},\delta_{x'}\in \MBD(\mu)$, then
\begin{equation}\label{diffform}\Tbar_{\mu}(\varphi)(x) - \Tbar_{\mu}(\varphi)(x') = \int_{\mathbb{R}^d}[K(y-x)-K(y-x')]\varphi(y)\,d\mu(y).
\end{equation}
(And so this formula holds for $m_d$-almost every $x$ and $x'$ in $\mathbb{R}^d$.)  To derive (\ref{diffform}), first recall that $\delta_x\in \MBD(\mu)$ means that $\tfrac{1}{|x-\cdot|^s}\in L^1_{\loc}(\mu)$, so the representation  (\ref{deltprodrep2}) yields that for any $\delta>0$,
\begin{equation}\begin{split}\nonumber\Tbar_{\mu}(\varphi)(x)- \Tbar_{\mu}(\varphi)(x') = \;&\Tbar_{\mu, \delta}(\varphi)(x) -\Tbar_{\mu, \delta}(\varphi)(x')\\
&+\int_{\R^d}[K^{\delta}(y-x)-K^{\delta}(y-x')]\varphi(y)d\mu(y).
\end{split}\end{equation}
But then (\ref{diffform}) follows from (\ref{deltdiffform}).\smallskip

$\bullet$ Now let $\varphi\in \Lip_B(\R^d)$, and $0<\delta\leq \Delta$.  On several occasions, we shall need to estimate the difference
$$F_{\delta,\Delta}(x) = \Tbar_{\mu,\delta}(\varphi)(x)-\Tbar_{\mu,\Delta}(\varphi)(x)$$
for $x\in \R^d$.  First note that due to the identity (\ref{deltaDelta}), we may write
\begin{equation}\begin{split}\label{Fdelt}
F_{\delta,\Delta}(x) &= \int_{\R^d}[K_{\delta}(y-x)-K_{\Delta}(y-x)]\varphi(y)d\mu(y)\\
& = \int_{B(x,\Delta)}[K_{\delta}(y-x)-K_{\Delta}(y-x)]\varphi(y)d\mu(y).
\end{split}\end{equation}
where the second equality follows from the first since the integrand vanishes if $|y-x|\geq \Delta$.

We can use this indentity to bound $F_{\delta,\Delta}(x)$ in absolute value:
\begin{equation}\label{Fabsval}
|F_{\delta,\Delta}(x)|\leq \int_{B(x,\Delta)}\frac{2|\varphi(y)|}{(\delta+|x-y|)^s}d\mu(y).
\end{equation}

The following lemma provides us with a simple but useful estimate for the function $F_{\delta,\Delta}$.

\begin{lem}\label{twodiffs} There is a constant $C>0$, depending on $s$ and $d$, such that for any $r>0$,
$$\int_{B(0,r)}\sup_{\delta\in (0,\Delta)}|F_{\delta,\Delta}(x)|dm_d(x)\leq C\|\varphi\|_{\infty}\min(\Delta,r)^{d-s}\mu(B(0,r+\Delta)).
$$
\end{lem}

To prove the inequality, note that by (\ref{Fabsval}), we have
\begin{equation}\begin{split}\nonumber\int_{B(0,r)}&\sup_{\delta\in (0,\Delta)}|F_{\delta,\Delta}(x)|dm_d(x)\\
&\leq 2\|\varphi\|_{\infty}\int_{B(0,r)}\int_{B(x,\Delta)}\frac{\,d\mu(y)}{|x-y|^s}dm_d(x),
\end{split}\end{equation}
from which the desired estimate follows from an application of Lemma \ref{locl1}.

We shall be especially interested in the functions $\Tbar_{\mu, \delta}(1)$ and $\Tbar_{\mu}(1)$ below.

\section{Reflectionless Measures}\label{refl}

A diffuse measure $\mu$ (with restricted growth at infinity) is called \textit{reflectionless} if
$$\langle T(f\mu), 1\rangle_{\mu}=0,
$$
for all $f\in \Lip_0(\mathbb{R}^d)$ satisfying $\int_{\R^d} f \,d\mu=0$.

In particular $\langle \T1(f\mu),1\rangle_{\mu}=1$ for any $f\in \Lip_0(\R^d)$.

We first note that, if $\mu$ is reflectionless, and $\nu \in \MD(\mu)$, then the value
$$\langle \T1\nu, 1\rangle_{\mu}
$$
is independent of the particular choice of smooth balancing measure $\nu_0 = \eta\mu$ with $\eta\in \Lip_0(\R^d)$ with $\int_{\mathbb{R}^d}\eta \,d\mu=1$.  To see this, pick another such balancing probability measure $\widetilde\eta \mu$.  Then by linearity, we see that
$$\langle T[\nu-\nu(\R^d)\eta\mu], 1\rangle_{\mu}-\langle T[\nu-\nu(\R^d)\widetilde\eta\mu], 1\rangle_{\mu} = \nu(\mathbb{R}^d)\langle T[(\widetilde\eta-\eta)\mu], 1\rangle_{\mu},
$$
and the right hand side is zero by the defining property of $\mu$ being reflectionless.

Now recall that, as a consequence of Lemma \ref{locl1}, $\delta_x \in \MBD(\mu)$ for $m_d$-almost every $x\in \R^d$.  For each such $x$, we have the formula (\ref{deltprodrep}) with $\varphi\equiv 1$, from which we conclude that the value of $\Tbar_{\mu, \delta}(1)(x)$ (and hence also $\Tbar_{\mu}(1)(x)$) does not depend on the choice of $\eta$.  However, the function $\Tbar_{\mu, \delta}(1)$ is continuous, so this property continues to hold for every $x\in \mathbb{R}^d$.

We shall frequently make use of the observation that if $\mu$ is a non-trivial reflectionless measure, then
\begin{equation}\label{inoutsame}\int_{\R^d}\Tbar_{\mu,\delta}(1)f\,d\mu = -\langle T^{\delta}(f\mu),1\rangle_{\mu}
\end{equation}
whenever $f\in \Lip_0(\R^d)$.  To see this, note that from (\ref{bardecentrep}) applied with the decent measure $\nu=f\mu$ and bounded Lipschitz function $\varphi \equiv 1$, the left hand side equals $ \langle \T1(f\mu),1\rangle_{\mu} - \langle T^{\delta}(f\mu), 1\rangle_{\mu} $,  but $\langle \T1(f\mu),1\rangle_{\mu}=0$.

\section{Nice measures}\label{nicesection}

From this point on, we shall assume that the CZ kernel $K$ under consideration is \textit{homogeneous} of order $-s$, that is,
$$K(\lambda x) = \lambda^{-s}K(x), \text{ for any }x\in \mathbb{R}^d\backslash \{0\} \text{ and }\lambda>0.
$$

For a homogeneous kernel $K$, we may perform a change of variable to motivate a natural condition that we shall frequently impose on a reflectionless measure.

Let $\nu$ be a signed measure, and $\varphi$ a function.  Then for $x\in \mathbb{R}^d$ and $r>0$, define $\nu_{x,r} = \tfrac{\nu(x+r\cdot)}{ r^s}$ and $\varphi_{x,r} = \varphi(x+r\cdot)$.   It is straightforward to verify that $T\nu(z) = T(\nu_{x,r})(\tfrac{z-x}{r})$ and $T_{\delta}\nu(z) = T_{\delta/r}(\nu_{x,r})(\tfrac{z-x}{r})$ for $\delta>0$  (and so $T^{\delta}(\nu)(z) = T^{\delta/r}(\nu_{x,r})(\tfrac{z-x}{r})$), whenever the integral defining the relevant potential converges.  Additionally if $\nu\in \MD(\mu)$ then
\begin{equation}\label{changeofvar}\langle T\nu,\varphi\rangle_{\mu} = \langle T(\nu_{x,r}),\varphi_{x,r}\rangle_{\mu(x+r\cdot)}=r^s\langle T(\nu_{x,r}),\varphi_{x,r}\rangle_{\mu_{x,r}},
\end{equation}
 for $\varphi\in \Lip_0(\mathbb{R}^d)$.

Thus
\begin{equation}\begin{split}\nonumber\Tbar_{\mu,\delta}(1)(z) &= \langle T_{\delta}(\delta_z-\eta\mu),1\rangle_{\mu} -\langle T^{\delta}(\eta\mu),1\rangle_{\mu}\\
& = r^s\langle T_{\delta/r}(\tfrac{1}{r^s}\delta_{\tfrac{z-x}{r}}-\eta_{x,r}\mu_{x,r}),1\rangle_{\mu_{x,r}} -r^s\langle T^{\delta/r}(\eta_{x,r}\mu_{x,r}),1\rangle_{\mu_{x,r}}\\
& = \langle T_{\delta/r}(\delta_{\tfrac{z-x}{r}}-[r^s\eta_{x,r}]\mu_{x,r}),1\rangle_{\mu_{x,r}} -\langle T^{\delta/r}([r^s\eta_{x,r}]\mu_{x,r}),1\rangle_{\mu_{x,r}}.
\end{split}\end{equation}

Now suppose that $\mu$ is a reflectionless measure.  Note that, if $\nu(\R^d)=0$ then $\nu_{x,r}(\R^d)=0$.   Thus, from (\ref{changeofvar}) we see that $\mu_{x,r}$ is also reflectionless.   Also, $\int_{\R^d}\eta\,d\mu=1=\int_{\R^d}[r^s\eta_{x,r}]d\mu_{x,r}$.    So,  since $[r^s\eta_{x,r}]\mu_{x,r}$ is an admissible balancing measure for the reflectionless measure $\mu_{x,r}$ (and the value of $\Tbar_{\mu_{x,r},\delta/r}(1)(\tfrac{z-x}{r})$ does not depend on the choice of the balancing measure), we have that $$\Tbar_{\mu,\delta}(1)(z) = \Tbar_{\mu_{x,r},\delta/r}(1)(\tfrac{z-x}{r}).$$

One consequence of these remarks is that in order to prove estimates for the potential $\Tbar_{\mu}(1)$ that are invariant under translations and dilations, it is natural to impose the hypothesis that $\mu_{x,r}(B(0,1))$ can be bounded uniformly in $x\in \R^d$ and $r>0$.   Fix $\Lambda>0$.  We say that a locally finite non-negative Borel measure $\mu$ is $\Lambda$-\textit{nice} if
$$\mu(B(x,r))\leq \Lambda r^s, \text{ for any }B(x,r)\subset\R^d,
$$
that is,  $\mu_{x,r}(B(0,1))\leq \Lambda$ for every $x\in \R^d$ and $r>0$.

Note that nice measures are diffuse (satisfy Definition \ref{pleasdef}), and satisfy the restricted growth at infinity assumption $\int_{|x|\geq 1}\tfrac{1}{|x|^{s+\alpha}}\,d\mu(x)<\infty$ regardless of $\alpha>0$.

For a nice measure $\mu$, we have a quantitative estimate on the continuity of $\Tbar_{\mu, \delta}(1)$.

\begin{lem}\label{Lipest} Suppose that $\mu$ is a $\Lambda $-nice measure.  There is a constant $C_1>0$, depending on $s$, $d$, $\alpha$, and $\|K\|_{*}$, such that for any $\delta>0$, and $y,y'\in \mathbb{R}^d$, $$|\Tbar_{\mu, \delta}(1)(y)-\Tbar_{\mu,\delta}(1)(y')|\leq \frac{C_1|y-y'|^{\alpha}}{\delta^{\alpha}}.$$\end{lem}

\begin{proof}  Recall the formula (\ref{diffform}):
$$\Tbar_{\mu, \delta}(1)(y)-\Tbar_{\mu,\delta}(1)(y')= \int_{\mathbb{R}^d}[K_{\delta}(x-y)-K_{\delta}(x-y')]\,d\mu(x).
$$
We shall integrate the absolute value of the integrand using the estimate
$$|K_{\delta}(x-y)-K_{\delta}(x-y')|\leq C \frac{|y-y'|^{\alpha}}{[\delta + \min(|x-y|, |x-y'|)]^{s+\alpha}}.
$$
Only in the case when both $\min(|x-y|,|x-y'|)<\tfrac{\delta}{2}$ and $|y-y'|<\tfrac{\delta}{2}$ is this estimate is not readily comparable to the previously stated bound (\ref{Kdeltest}).\footnote{To see this, note that if $\min(|x-y|,|x-y'|)\geq\tfrac{\delta}{2}$, then the claimed estimate is at least a constant multiple of $\tfrac{|y-y'|^{\alpha}}{ \min(|x-y|, |x-y'|)]^{s+\alpha}}$, while if $|y-y'|\geq\delta/2$ and $\min(|x-y|,|x-y'|)<\tfrac{\delta}{2}$ then the claimed bound is least a constant multiple of $\delta^{-s}$.} However, under these assumptions both $|x-y|$ and $|x-y'|$ are no greater than $\delta$.  Hence
\begin{equation}\begin{split}\nonumber|K_{\delta}(x-y)-K_{\delta}(x-y')| & = \frac{|K(x-y)|x-y|^{s+\alpha}-K(x-y')|x-y'|^{s+\alpha}|}{\delta^{s+\alpha}}\\
&\leq \frac{\|K\|_{*}|y-y'|^{\alpha}}{\delta^{s+\alpha}}.
\end{split}\end{equation}

The estimation of the integral is now a routine exercise:
\begin{equation}\begin{split}\nonumber
\int_{\R^d} &\frac{|y-y'|^{\alpha}}{[\delta + \min(|x-y|, |x-y'|)]^{s+\alpha}}d\mu(x)\\
&\leq C|y-y'|^{\alpha}\Bigl[\int_{0}^{\infty}\frac{\mu(B(y,t))}{(\delta+t)^{s+\alpha}}\frac{dt}{t}+\int_{0}^{\infty}\frac{\mu(B(y',t))}{(\delta+t)^{s+\alpha}}\frac{dt}{t}\Bigl]\\
&\leq C |y-y'|^{\alpha}\int_0^{\infty}\frac{t^s}{(\delta+t)^{s+\alpha}}\frac{dt}{t} = \frac{C |y-y'|^{\alpha}}{\delta^{\alpha}}\int_0^{\infty}\frac{t^s}{(1+t)^{s+\alpha}}\frac{dt}{t}.
\end{split}\end{equation}
But the integral in this final line is clearly convergent, and this proves the lemma.
\end{proof}

The next estimate provides with an estimate for $\langle T (f\mu), \varphi\rangle_{\mu}$ in terms of the supports of $\supp(f)$ and $\supp(\varphi)$, provided that $\mu$ is nice.  Recall that, for a set $E$, the closed $\delta$-neighbourhood of $E$ is denoted by $E_{\delta}$.

\begin{lem}\label{lipgenest} Suppose that $\mu$ is a $\Lambda$-nice measure.   Let $f\in \Lip_0(\R^d)$, and $\varphi\in \Lip_B(\R^d)$.  Then, for $\delta>0$,
\begin{equation}\begin{split}\nonumber |\langle T^{\delta}(f\mu),\varphi\rangle_{\mu}|\leq C_2\delta\bigl[\|f\|_{\Lip}\|\varphi\|_{\infty} &+ \|f\|_{\infty}\|\varphi\|_{\Lip}\bigl]\\
&\cdot\mu(\supp(f)\cap[\supp(\varphi)]_{\delta}).
\end{split}\end{equation}
\end{lem}

\begin{proof} Set $E=\supp(f)$ and $F=\supp(\varphi)$.  Recall that $H_{f,\varphi}(x,y) = \tfrac{1}{2}\bigl[f(y)\varphi(x) - f(x)\varphi(y)\bigl]$.  Thus,
$$\|H_{f,\varphi}\|_{\Lip}\leq \bigl[\|f\|_{\Lip}\|\varphi\|_{\infty}+ \|f\|_{\infty}\|\varphi\|_{\Lip}\bigl].
$$
Now,
\begin{equation}\begin{split}\nonumber|\langle & T^{\delta}(f\mu),\varphi\rangle_{\mu}|=\Bigl|\iint\limits_{|x-y|< \delta}H_{f,\varphi}(x,y)K^{\delta}(x-y)d\mu(x)\,d\mu(y)\Bigl|\\
&\leq \|H_{f,\varphi}\|_{\Lip}\iint\limits_{|x-y|< \delta}\frac{\chi\ci{E}(x)\chi\ci{F}(y)+\chi\ci{F}(x)\chi\ci{E}(y)}{|x-y|^{s-1}}d\mu(x)d\mu(y)\\
&=2 \|H_{f,\varphi}\|_{\Lip}\iint\limits_{|x-y|< \delta}\frac{\chi\ci{E}(x)\chi\ci{F}(y)}{|x-y|^{s-1}}d\mu(x)d\mu(y).
\end{split}\end{equation}
But if $x\in E$, $y\in F$, and $|x-y|< \delta$, then clearly $x\in F_{\delta}$.  Thus, we may write
\begin{equation}\begin{split}\nonumber\iint\limits_{|x-y|< \delta}\frac{\chi\ci{E}(x)\chi\ci{F}(y)}{|x-y|^{s-1}}d\mu(x)d\mu(y)&=\iint\limits_{|x-y|< \delta}\frac{\chi\ci{E\cap F_{\delta}}(x)\chi\ci{F}(y)}{|x-y|^{s-1}}d\mu(x)d\mu(y)\\
& \leq\int_{E\cap F_{\delta}} \int_{B(x,\delta)}\frac{1}{|x-y|^{s-1}}d\mu(y)d\mu(x).
\end{split}\end{equation}
However, for every $x\in \R^d$,
$$\int_{B(x,\delta)}\frac{1}{|x-y|^{s-1}}d\mu(y)\leq C\int_0^{\delta}\frac{\mu(B(x,r))}{r^{s-1}}\frac{dr}{r}\leq C\Lambda\delta.
$$
Bringing everything together, we get
$$|\langle  T^{\delta}(f\mu),\varphi\rangle_{\mu}|\leq C \|H_{f,\varphi}\|_{\Lip}\Lambda  \delta \mu(E\cap F_{\delta}),
$$
as required.
\end{proof}

\begin{comment}\begin{lem}\label{lipball}  Suppose that $\mu$ is a $\Lambda$-nice measure.  Fix a ball $B= B(x_0,r)$.  If $f,\varphi\in \Lip_0(B)$, then for any $\delta >0$,
\begin{equation}\begin{split}\nonumber |\langle T^{\delta}(f\mu),\varphi\rangle_{\mu}|&\leq C_2r\bigl[\|f\|_{\Lip}\|\varphi\|_{\infty}+ \|f\|_{\infty}\|\varphi\|_{\Lip}\bigl]\mu(B)\\
&\leq 2C_2r^2\|f\|_{\Lip}\|\varphi\|_{\Lip}\mu(B).
\end{split}\end{equation}
\end{lem}

\begin{proof}  Since $f$ and $\varphi$ are supported in $B$, the second inequality follows from the first.  To prove the first inequality, recall that $H_{f,\varphi}(x,y) = \tfrac{1}{2}\bigl[f(y)\varphi(x) - f(x)\varphi(y)\bigl]$.  Thus,
$$\|H_{f,\varphi}\|_{\Lip}\leq 2\bigl[\|f\|_{\Lip}\|\varphi\|_{\infty}+ \|f\|_{\infty}\|\varphi\|_{\Lip}\bigl].
$$
Also, $\supp(H_{f,\varphi})\subset B\times B$, and so
$$|\langle T^{\delta}(f\mu),\varphi\rangle_{\mu}|\leq C\|H_{f,\varphi}\|_{\Lip}\int_{B\times B}\frac{1}{|z-y|^{s-1}}\,d\mu(y)\,d\mu(z).
$$
On the other hand,
$$\int_{B\times B}\frac{1}{|z-y|^{s-1}}\,d\mu(y)\,d\mu(z)\leq C\int_{B}\int_0^{2r} \frac{\mu(B(z,t))}{t^{s-1}}\frac{dt}{t}\,d\mu(z)\leq C\Lambda r\mu(B),
$$
which completes the proof.
\end{proof}
\end{comment}

Our final estimate of this section is a consequence of the previous lemma.  It will play an important role in what follows.

\begin{lem}\label{lipanest}  Suppose that $\mu$ is a $\Lambda$-nice measure.   Fix a point $x\in \R^d$, a radius $r>0$, and a gauge $\delta>0$.  If $f\in \Lip_0(B(x, r+2\delta))$ satisfies $0\leq f\leq 1$ in $\mathbb{R}^d$, $f\equiv 1$ on $B(x,r+\delta)$, and $\|f\|_{\Lip}\leq \tfrac{A}{\delta}$ for some $A>1$, then
$$|\langle T^{\delta}(f\mu),1\rangle_{\mu}|\leq C_3A\mu(B(x,r+2\delta)\backslash B(x,r)).
$$
\end{lem}

\begin{proof}
Using the anti-symmetry of the kernel $K^{\delta}$ we see that
$$\langle T^{\delta}(f\mu),1\rangle_{\mu}= \langle T^{\delta}(f\mu),1-f \rangle_{\mu}.$$
We now apply Lemma \ref{lipgenest} with $\varphi=1-f$.  Since $\supp(f)\cap [\supp(\varphi)]_{\delta}\subset B(x,r+2\delta)\backslash B(x,r)$, $\|\varphi\|_{\Lip}\leq \tfrac{A}{\delta}$, and $0\leq \varphi\leq 1$ on $\R^d$, the stated estimate follows immediately.
\end{proof}

\section{The Collapse Lemma}\label{colsec}

This section is devoted to introducing the main technical tool of the paper.  Throughout the section, suppose that $\mu$ is a non-trivial $\Lambda$-nice reflectionless measure.

For a unit vector $\e\in \mathbb{C}^{d'}$, and $\eps>0$, define
$$E(\e, \eps, r) = \bigl\{ x\in \mathbb{R}^d: \Re[\e\cdot \Tbar_{\mu,\delta}(1)](x)>\eps \text{ for all }\delta\in (0,r)\bigl\}.
$$

\begin{prop}[The Collapse Lemma] \label{collem} Let $\eps\in (0,\tfrac{1}{2})$.
  There exists $\beta>0$ (depending on  $s$ and $\alpha$), such that if $\kap \leq \kap(\eps)=c_{9}\eps^{\beta}$, then the following holds:  If $E(\e, \eps, r)$ is $\kap r$-dense in $B(x_0,2r)$, then $\mu(B(x_0,r))=0$.
\end{prop}

We shall sometimes refer to $\kap$ as the \textit{abundancy parameter}, as it governs the abundance of the set $E(\e,\eps,1)$ in the ball $B(x_0,2r)$.

First note that by considering the measure $\tfrac{\mu(x_0+r\cdot)}{r^s}$ instead of $\mu$, it suffices to prove the result for $x_0=0$ and $r=1$.  The proof relies upon two ideas, which are expressed by the following two lemmas.

\begin{lem}\label{decaylem}  Let $\eps\in (0,\tfrac{1}{2})$, $\kap\in (0,1)$, and $t\in (1,2]$.  Suppose that $E(\e,\eps, 1)$ is $\kap $-dense in $B(0,t)$.  If $\eps  \geq 2C_1\kap^{\tfrac{\alpha}{2}}$, then $$\mu(B(0,t-\sqrt{\kap}))\leq (1-\lambda)\mu(B(0,t)),$$ with $\lambda=c_4\eps$.
\end{lem}

We remark that $\eps>0$ is to be considered a fixed noticeable quantity, and so the lemma says that as long the abundancy parameter $\kap$  is small, the measure of the slightly smaller ball $\mu(B(0,t-\sqrt{\kap}))$ is noticeably less than $\mu(B(0,t))$.

The idea of the proof is quite simple.  Since $\Tbar_{\mu, \sqrt{\kap}}(1)$ is essentially constant on scale $\kap$, we have $\Re[\e\cdot\Tbar_{\mu, \sqrt{\kap}}(1)]>\tfrac{\eps}{2}$ on $B(0,t-\sqrt{\kap})$.  The reflectionless property means, roughly speaking, that $\Tbar_{\mu}(1)$ vanishes on $\supp(\mu)$, so on the set $E=\supp(\mu)\cap B(0, t-\sqrt{\kap})$ we must have $\Re[\e\cdot T_{\mu}^{\sqrt{\kap}}(1)]=\Re[\e\cdot T_{\mu}^{\sqrt{\kap}}(\chi_{B(0,t)})] <-\tfrac{\eps}{2}$.  The antisymmetry of the kernel $K$ implies, however, that the average of $T_{\mu}^{\sqrt{\kap}}(\chi_{B(0,t-\sqrt{\kap})})$ over $E$ with respect to $\mu$ is $0$.  Hence the contribution of the rim $B(0,t)\backslash B(0,t-\sqrt{\kap})$ should be noticeable, which forces its $\mu$-measure to be a noticeable portion of $\mu(B(0,t))$.  Of course, $\Tbar_{\mu}(1)$ is defined on $\supp(\mu)$ only in the sense of a linear form on $\Lip_0(\mathbb{R}^d)$, so we have to use Lipschitz cutoff functions instead of rough characteristic functions.

\begin{proof}
Choose a non-negative function $f\in \Lip_0(B(0, t))$, with $f \equiv 1$ on $B(0, t-\tfrac{1}{2}\sqrt{\kap})$, $0\leq f\leq 1$ on $\R^d$, and $\|f\|_{\Lip}\leq \tfrac{C}{\sqrt{\kap}}$.

For every $x\in B(0, t)$, there exists $x'\in E(\e, \eps, 1)$ with $|x-x'|\leq\kap$.  Thus, for any $\delta\in[ \sqrt{\kap},1)$, \begin{equation}\label{kapholder}|\Tbar_{\mu,\delta}(1)(x) - \Tbar_{\mu,\delta}(1)(x')|\leq C_1\Bigl(\frac{\kap}{\delta}\Bigl)^{\alpha}\leq C_1\kap^{\tfrac{\alpha}{2}}.\end{equation} In the case $\delta=\sqrt{\kap}$, we infer from (\ref{kapholder})  that $\Re[\e\cdot \Tbar_{\mu,\sqrt{\kap}}(1)](x)>\eps-C_1\kap^{\tfrac{\alpha}{2}}\geq\tfrac{\eps}{2}$.

As a result of this property and the reflectionlessness of $\mu$, we have
\begin{equation}\nonumber\begin{split}\frac{\eps}{2} \mu(B(0, t-\sqrt{\kap}))&\leq \frac{\eps}{2}\int_{\mathbb{R}^d} f\,d\mu\\
&\leq \Re\Bigl[ \e\cdot\int_{\mathbb{R}^d}\Tbar_{\mu,\sqrt{\kap}}(1)f \,d\mu\Bigl] \stackrel{(\ref{inoutsame})}{=} -\Re[\e\cdot  \langle T^{\sqrt{\kap}}(f\mu),1\rangle_{\mu}]. \\
\end{split}\end{equation}
On the other hand, Lemma \ref{lipanest}, applied with the point $0$, radius $t-\sqrt{\kap}$, and gauge $\delta = \tfrac{1}{2}\sqrt{\kap}$, yields
$$|\langle T^{\sqrt{\kap}}(f\mu),1\rangle_{\mu}|\leq C\mu(B(0, t)\backslash B(0, t-\sqrt{\kap})).
$$

Bringing these two estimates together, we see that
\begin{equation}\label{rearrange}\eps\mu(B(0,t-\sqrt{\kap})) \leq C\mu(B(0,t)\backslash B(0,t-\sqrt{\kap})).
\end{equation}
From which it follows that,
\begin{equation}\label{decay}\mu(B(0, t-\sqrt{\kap})) \leq (1-\lambda) \mu(B(0,t)),
\end{equation}
with  $\lambda = c_4\eps$, for $c_4$ chosen suitably.\end{proof}

Lemma \ref{decaylem} goes a long way towards the proof of the Collapse lemma because it implies that the measure of $B(0,1)$ can be made an arbitrarily small portion of $\mu(B(0,2))$ if $\kap$ is chosen small enough.  However, it cannot finish the job alone because the abundance parameter $\kap$ doesn't change along the way and we have to subtract $\sqrt{\kap}$ from the radius at each step.  We would like to gradually diminish the abundance parameter as we go.

The next lemma shows that it is, indeed, possible to diminish the abundancy parameter once $\mu(B(0,t))$ becomes small enough and, moreover, the abundancy parameter for the smaller ball $B(0,t-\sqrt{\kap})$ can be chosen as a power of the measure $\mu(B(0,t))$.  This comes at the cost of slightly decreasing the size parameter $\eps$, but since the decay of the measure is geometric, we may then hope to be able to make infinitely many steps and bring the measure to $0$ before the radius or $\eps$ reduces to $0$.

\begin{lem}\label{densinclem} Let $\eps\in (0, \tfrac{1}{2})$, $\kap\in (0, 1)$,  $m\in (0, 1)$, and $t\in (1,2]$.  Suppose that $E(\e,\eps,1)$ is $\kap$-dense in $B(0,t-\sqrt{\kap})$, and $\mu(B(0,t))\leq m$.  There exists a constant $C_6$ such that for
$$\eps'=\eps - C_1[\kap^{\tfrac{\alpha}{2}}+\sqrt{m}], \, \kap^{\,\prime} = C_6 m^{\tfrac{1}{2d}},\,\text{ and } t'=t-\sqrt{\kap},
$$
the intersection $E(\e, \eps',1)$ is $\kap^{\,\prime}$-dense in $B(0,t')$. \end{lem}

The proof is essentially a combination of Lemma \ref{twodiffs} with Chebyshev's inequality.  The combination of these two simple tools tells us that the Lebesgue measure of the set where $\sup_{\delta\in (0,\sqrt{\kap})}| \Tbar_{\mu, \delta}(1)(x) - \Tbar_{\mu, \sqrt{\kap}}(1)(x)|$ is noticeable in $B(0,t-\sqrt{\kap})$ is controlled in terms of the measure $\mu(B(0,t))$.  On the other hand, much like in the proof of Lemma \ref{decaylem}, the abundancy hypothesis in terms of $\kap$ ensures that $\Re[\e\cdot\Tbar_{\mu, \sqrt{\kap}}(1)]$ is never much smaller than $\eps$ in $B(0,t-\sqrt{\kap})$.  Combining the two facts with a decent choice of parameters shows that $E(\e,\eps',1)$ has almost full Lebesgue measure in $B(0,t-\sqrt{\kap})$, and therefore must be very abundant.

\begin{proof}For any $x\in B(0,t')$, there exists $x'\in E(\e,\eps, 1)$ such that  $|x-x'|\leq\kap$.
By writing $\e\cdot\Tbar_{\mu,\delta}(1)(x) = \e\cdot\Tbar_{\mu,\delta}(1)(x')+  \e\cdot[\Tbar_{\mu,\delta}(1)(x) - \Tbar_{\mu,\delta}(1)(x')]$, we see from (\ref{kapholder}) that \begin{equation}\begin{split}\label{denslowbd} \Re[\e\cdot\Tbar_{\mu,\delta}(1)](x) > \eps-C_1\kap^{\tfrac{\alpha}{2}}=\eps'+C_1\sqrt{m},
\end{split}\end{equation}
for any $\delta\in[\sqrt{\kap}, 1)$.

Set $F_{\delta,\sqrt{\kap}}(x) = \Tbar_{\mu, \delta}(1)(x) - \Tbar_{\mu, \sqrt{\kap}}(1)(x)$.  From (\ref{denslowbd}), we infer that if $\Re[\textbf{e}\cdot\Tbar_{\mu,\delta}(1)](x) <\eps'$ for some $x\in B(0,t')$ and $\delta\in (0,1)$, then $\delta<\sqrt{\kap}$ and $|F_{\delta,\sqrt{\kap}}(x)|>C_1\sqrt{m}$ (the second condition follows since $\Re[\e\cdot\Tbar_{\mu,\sqrt{\kap}}(1)](x)>\eps'+C_{1}\sqrt{m}$, and certainly $\Re[\e\cdot\Tbar_{\mu,\delta}(1)](x)\geq \Re[\e\cdot\Tbar_{\mu,\sqrt{\kap}}(1)](x)-|F_{\delta,\sqrt{\kap}}(x)|$).

Now note that Lemma \ref{twodiffs} yields
\begin{equation}\begin{split}\nonumber\int_{B(0, t - \sqrt{\kap})}\sup_{\delta\in (0,\sqrt{\kap})}|F_{\delta,\sqrt{\kap}}(x)|dm_d(x)&\leq C_5\mu(B(0, t)) \kap^{\tfrac{d-s}{2}}\leq C_5m.
\end{split}\end{equation}

Consequently, Chebyshev's inequality yields that
\begin{equation}\begin{split}\nonumber m_d\bigl(\bigl\{x\in B(0, t')\,: \sup_{\delta \in (0,\sqrt{\kap})}|F_{\delta,\sqrt{\kap}}(x)|>C_1\sqrt{m}\bigl\}\bigl) \leq \frac{2C_5}{C_1}\sqrt{m}.\end{split}\end{equation}

Now, fix $C_6  \geq \bigl(\tfrac{2^{d+1}C_5}{\omega_d C_1}\bigl)^{\tfrac{1}{d}}$, where $\omega_d$ denotes the volume of the $d$-dimensional unit ball. Then the set $E(\e,\eps', t')\cap B(0,t')$ is $\kap^{\,\prime}$-dense in $B(0, t')$ as long as $\kap^{\,\prime}<\tfrac{1}{4}$.  Indeed, $m_d(B(0,t')\backslash E(\e, \eps', t'))<\omega_d\bigl(\tfrac{\kap^{\,\prime}}{2}\bigl)^d$.  But, if for some $x\in B(0,t')$, the distance from $x$ to $E(\e, \eps', t')\cap B(0,t')$ is greater than $\kap^{\,\prime}$, then there is a ball of radius $\tfrac{\kap^{\,\prime}}{2}$ that is contained in $B(0,t')$ but disjoint from $E(\e,\eps',t')$.  The existence of this ball is in contradiction with the measure estimate.  If $\kap^{\,\prime}\geq \tfrac{1}{4}$, then $\kap^{\,\prime}\geq \kap$, so there is nothing to prove.
\end{proof}

We now combine Lemmas \ref{decaylem} and \ref{densinclem} to prove the Collapse Lemma.  Before giving the formal proof we outline the idea.   Note that $\eps>0$ is fixed and the starting measure $\mu(B(0,2))\leq \Lambda 2^s$.  Our only freedom is in the choice of the starting value of $\kap>0$.

We will iterate Lemma \ref{decaylem} first to reduce the measure $\mu(B(0,t))$ to a sufficiently small value $m_0>0$.  Regardless of the choice of $\kap$, this will require $N\approx \tfrac{1}{\eps}\log(\tfrac{1}{m_0})$ steps as long as the radius does not collapse, which can be ensured by choosing $\kap$ so small that $N\sqrt{\kap}<\tfrac{1}{2}$ in addition to the requirements of Lemma \ref{decaylem}.

Once the measure is small, we start iterating Lemma \ref{densinclem} alternatingly with Lemma \ref{decaylem}.  In this case,  $\eps>0$ starts to decay as well from its initial value.  The dynamics of the parameters $t$, $m=\mu(B(0,t))$, and $\eps$, that arises is
$$t_{j+1}=t_j - \sqrt{\kap_j},\; m_{j+1}=(1-c_4\eps_j)m_j,\text{ and }\eps_{j+1} = \eps_j-C_1[\kap_j^{\tfrac{\alpha}{2}}+\sqrt{m_j}],
$$
with $\kap$ related to $m$ by $\kap_j = C_6m_{j-1}^{\tfrac{1}{2d}}$.

Our main task is to be able to make infinitely many steps while $t_j$ stays above $1$, and $\eps_j$ stays above $\tfrac{\eps}{2}$, say.  Under these conditions, $m_j \leq (1-\tfrac{c_4}{2}\eps)^jm_0$, so the quantities responsible for the deterioration of $t_j$ and $\eps_j$ from step to step have a fixed geometric decay and a factor of $m_0$ in them.   Thus, if $m_0>0$ is chosen small enough, the sum of these quantities after arbitrarily many steps during which $\eps_j>\tfrac{\eps}{2}$ and $t_j>1$ will be  very small too, which will allow us to always make the next step without breaking through the corresponding barriers.  We now turn to the details.

\begin{proof}[Proof of Proposition \ref{collem}]  Fix $m_0 >0$ to be chosen later, and suppose that $\mu(B(0,\tfrac{3}{2}))\leq m_0$.  Set $t_0 = \tfrac{3}{2}$, $\kap_0 =\kap$, and $\eps_0=\eps$.  Then $E(\e, \eps_0, 1)$ is $\kap_0$-dense in $B(0,t_0)$ by the hypotheses of  Proposition \ref{collem} (provided that $c_9$ is chosen to be less than $1$).
For $j\geq 1$, set
$$\eps_{j} =\eps_0 -\sum_{\ell=0}^{j-1}C_1\bigl[\kap_{\ell}^{\tfrac{\alpha}{2}}+\sqrt{m_{\ell}}\bigl],\; \kap_j = C_{6}m_{j-1}^{\tfrac{1}{2d}},\; t_j = t_0 - \sum_{\ell=0}^{j-1}\sqrt{\kap_{\ell}},
$$
and $m_j = (1-\tfrac{\lambda}{2})m_{j-1}$, with $\lambda=c_4\eps$ as in Lemma \ref{decaylem}.

Suppose that for some $j\geq 0$, $E(\e, \eps_j, 1)$ is $\kap_j$-dense in $B(0,t_j)$, and also that $\mu(B(0,t_j))\leq m_j$.   If \begin{equation}\label{epsjkapj}\eps_j \geq \frac{\eps}{2}, \;\; 2C_1\kap_j^{\tfrac{\alpha}{2}}\leq \frac{\eps}{2},\; \text{ and }t_j>1,\end{equation}then $2C_1\kap_j^{\tfrac{\alpha}{2}}\leq \eps_j$, and Lemma \ref{decaylem} yields $\mu(B(0,t_{j+1}))\leq (1-c_4\eps_j)m_j$. But since $\eps_j\geq \tfrac{\eps}{2}$, we have $c_4\eps_j \geq \tfrac{\lambda}{2}$, and so $\mu(B(0,t_{j+1}))\leq m_{j+1}$.  

On the other hand, Lemma \ref{densinclem} ensures that $E(\e, \eps_{j+1},1)$ is $\kap_{j+1}$-dense in $B(0,t_{j+1})$.

Bringing these two observations together, we see that if (\ref{epsjkapj}) holds for each $j\geq 0$, then $$\mu(B(0,t_{j}))\leq \Bigl(1-\frac{\lambda}{2}\Bigl)^jm_0 \text{ for every }j\geq 0,$$
and so $\mu(B(0,1))=0$, which is the desired conclusion of the Collapse Lemma.

We shall now make a choice of parameters to ensure that (\ref{epsjkapj}) is valid.  Our requirements that $\eps_j \geq \tfrac{\eps}{2}$ and  $2C_1\kap_j^{\tfrac{\alpha}{2}}\leq\tfrac{\eps}{2}$ for every $j$ will be satisfied if
$$C_1\kap^{\tfrac{\alpha}{2}}+\sum_{\ell=0}^{\infty}C_1\Bigl[C_6^{\tfrac{\alpha}{2}}\Bigl(1-\frac{\lambda}{2}\Bigl)^{\tfrac{\alpha \ell}{4d}}m_0^{\tfrac{\alpha}{4d}}+\Bigl(1-\frac{\lambda}{2}\Bigl)^{\tfrac{\ell}{2}}\sqrt{m_0}\Bigl]<\frac{\eps}{4} \text{ and }C_1C_6m_0^{\tfrac{1}{2d}}<\frac{\eps}{4}.
$$
On the other hand, $t_j>1$ for all $j\geq 1$ if $$\sum_{\ell=0}^{\infty}\sqrt{C_6}\Bigl(1-\frac{\lambda}{2}\Bigl)^{\tfrac{\ell}{4d}}m_0^{\tfrac{1}{4d}}<\frac{1}{2}.
$$

Notice that $\sum_{\ell=0}^{\infty}\bigl(1-\tfrac{\lambda}{2}\bigl)^{\tfrac{\alpha \ell}{4d}}\leq \tfrac{C}{\lambda}\leq \tfrac{C}{\eps}$.  Therefore, if we choose $m_0 = c_{7}\eps^{\gamma}$ for suitable constants $c_{7}>0$ and $\gamma = \gamma(d,s,\alpha)>0$, then the inequalities comprising (\ref{epsjkapj}) are satisfied provided that $\kap<\bigl(\tfrac{\eps}{4C_1}\bigl)^{\tfrac{2}{\alpha}}$.

It remains to ensure that $\mu(B(0,t_0)) = \mu(B(0,\tfrac{3}{2}))\leq m_0$.   Fix $N\in \mathbb{N}$.  
A repeated application of Lemma \ref{decaylem}  yields
\begin{equation}\nonumber\begin{split} \mu(B(0, 2-N\sqrt{\kap})) \leq (1-\lambda)^N\mu(B(0,2))\leq (1-\lambda)^N\Lambda2^s.
\end{split}\end{equation}
If $N\sqrt{\kap}<\tfrac{1}{2}$, then $\mu(B(0,\tfrac{3}{2}))\leq (1-\lambda)^N\Lambda2^s.$  Thus, it suffices to ensure that $(1-\lambda)^N \leq \tfrac{m_0}{\Lambda 2^s}$. This condition dictates our choice of $N$ as  $N= \lfloor C_{8}\tfrac{\log\tfrac{1}{\eps}}{\eps}\rfloor+1.$ All that is left is to choose $\kap(\eps)$.  The two assumptions we need to satisfy are
$$\kap(\eps)<\Bigl(\frac{\eps}{4C_6}\Bigl)^{\tfrac{2}{\alpha}}, \text{ and }\kap(\eps)<\frac{\eps^2}{4 C_{8}^2\log^2\tfrac{1}{\eps}} \Bigl(\approx \frac{1}{4N^2}\Bigl).
$$
So  we can choose $\kap(\eps)=c_{9}\eps^{\beta}$, for suitable $c_{9}>0$ and $\beta=\beta(s,\alpha)>0$.\end{proof}

\subsection{Consequences of the Collapse Lemma}

The remainder of the section is devoted to consequences of the Collapse lemma.  Again, fix $\mu$ to be a non-trivial $\Lambda$-nice reflectionless measure.  We begin with a simple alternative:

\begin{lem}\label{colalt}  For each $\eps>0$, there exist $M=M(\eps)>0$ and $\tau = \tau(\eps)>0$, such that whenever $|\Tbar_{\mu, Mr}(1)(x)|>\eps$ for some $x\in \mathbb{R}^d$ and $r>0$, one of the following two statements must hold:

(i) $\mu(B(x,2Mr))\geq \tau r^s$, or

(ii)  $\mu(B(x,r))=0$.
\end{lem}

\begin{proof}  We may assume that $x=0$ and $r=1$.  Fix $\tau>0$, and $M>4$.  Suppose that $\mu(B(0,2M))\leq\tau$.  For $\delta>0$, set $F_{\delta,M} = \Tbar_{\mu, \delta}(1) - \Tbar_{\mu, M}(1)$.  Then by Lemma \ref{twodiffs},$$\int_{B(0,2)} \sup_{\delta\in (0,M)} \bigl|F_{\delta,M}(y)\bigl| dm_d(y) \leq C_{10} \tau.
$$
Consequently, the Chebyshev inequality ensures that the set $E = \bigl\{y\in B(0,2): \sup_{\delta\in(0,M)} |F_{\delta,M}(y)|<\tfrac{\eps}{4}\bigl\}$ is $C_{11}\bigl(\tfrac{\tau}{\eps}\bigl)^{\tfrac{1}{d}}$-dense in $B(0,2)$ (cf. the proof of Lemma \ref{densinclem}).

Set $\e$ to be the unit vector satisfying $\e\cdot \Tbar_{\mu, M}(1)(0) = |\Tbar_{\mu, M}(1)(0)|$.  Suppose $y\in E$, and $\delta\in (0,1)$.  Write
$$\Tbar_{\mu,\delta}(1)(y) = \Tbar_{\mu,M}(1)(0) + F_{\delta,M}(y) + [\Tbar_{\mu,M}(1)(y)- \Tbar_{\mu,M}(1)(0)].
$$
Since $|\Tbar_{\mu,M}(1)(y)- \Tbar_{\mu,M}(1)(0)|\leq\tfrac{2C_1}{M^{\alpha}}$, we infer from the above equality that
$\Re [\e\cdot\Tbar_{\mu, \delta}(1)](y)>\tfrac{3\eps}{4} - \frac{2C_1}{M^{\alpha}}.$  This quantity is at least $\tfrac{\eps}{2}$ if $M\geq M(\eps)= \bigl(\tfrac{8C_1}{\eps}\bigl)^{\tfrac{1}{\alpha}}$.   If $C_{11}\bigl(\tfrac{\tau}{\eps}\bigl)^{1/d}\leq \kap\bigl(\min\bigl[\tfrac{\eps}{2},\tfrac{1}{2}\bigl]\bigl)$, the Collapse Lemma implies that $\mu(B(0,1))=0$.  Thus, the alternative holds with $\tau=c_{12}\eps\kap^d$ for a suitable constant $c_{12}>0$.
\end{proof}

\begin{cor}\label{collapsecor1}  For each $\eps\in(0,\tfrac{1}{2})$, there exist $M'=M'(\eps)>0$ and $\tau=\tau'(\eps)>0$, such that if $|\Tbar_{\mu}(1)(x)|>\eps$, and $\dist(x, \supp(\mu))=r$, then $\mu(B(x,M'r))\geq \tau' r^s.$\end{cor}

\begin{proof}  Without loss of generality, we may assume that $r=1$ and $x=0$.   Set $M=M\bigl(\tfrac{\eps}{2}\bigl)$ as in Lemma \ref{colalt}, and fix $M'=4M$.   By a trivial absolute value estimate, $|\Tbar_{\mu, 2M}(1)(0)|>\eps-\int_{B(0,2M)\backslash B(0,\tfrac{1}{2})}\tfrac{1}{|y|^s}\,d\mu(y)>\eps-C\mu(B(0, 2M))$.  So $|\Tbar_{\mu, 2M}(1)(0)|>\tfrac{\eps}{2}$ if $\mu(B(0,2M))\leq \sigma= c\eps$ for a sufficiently small constant $c>0$.  But now the assumptions of Lemma \ref{colalt} are satisfied at the point $x=0$, radius $r=2$, and with $\eps$ replaced by $\tfrac{\eps}{2}$.  By hypothesis $\mu(B(0,2))>0$, so $\mu(B(0,4M))>\tau$, where $\tau=\tau\bigl(\tfrac{\eps}{2}\bigl)$ is given by Lemma \ref{colalt}.  Setting $\tau' = \min\bigl[\sigma, \tau\bigl]$ completes the proof.
\end{proof}

\begin{cor}\label{lebcor}  $\Tbar_{\mu}(1)(x)=0$ for $m_d$-almost every $x\in \supp(\mu)$.
\end{cor}

\begin{proof}  By standard measure theory, the limit $D(x)=\lim_{r\to 0} \tfrac{\mu(B(x,r))}{r^d}$ exists and is finite for $m_d$-almost every $x\in \mathbb{R}^d$.  It therefore suffices to prove that if $|\Tbar_{\mu}(1)(x)|>2\eps$ for some $\eps>0$, and $D(x)$ exists and is finite, then $x\not\in \supp(\mu)$.  Set $M=M(\eps)$, and $\tau=\tau(\eps)$, as in Lemma \ref{colalt}.  If $D(x)<\infty$, then $\mu(B(x,r))\leq (D(x)+1)r^d$ for all sufficiently small $r$.   Thus $\delta_x\in \MBD(\mu)$, and moreover provided that $r$ is sufficiently small, $\int_{B(x,Mr)}\tfrac{1}{|y-x|^s}\,d\mu(y) \leq C[D(x)+1](Mr)^{d-s}\leq \eps$.  But then both $|\Tbar_{\mu, Mr}(1)(x)|>\eps$ and $\mu(B(x,2Mr))\leq (D(x)+1)r^d\leq \tau r^s$ for small enough $r$.  From Lemma \ref{colalt}, we infer that $\mu(B(x,r))=0$.   So $x\not\in \supp(\mu)$.
\end{proof}

\subsection{Porosity}The final result of this section is a porosity property in balls where $\Tbar_{\mu}(1)$ is large on average.  This will serve as the primary tool in showing that the support of a reflectionless measure for the $s$-Riesz transform is nowhere dense, which shall be proved in Part II.

\begin{lem}\label{intpor}  For each $\eps>0$, there exists $\lambda=\lambda(\eps)>0$, such that if $\int_{B(x,r)}|\Tbar_{\mu}(1)(y)| dm_d(y)>\eps m_d(B(x,r))$, then there is a ball $B'\subset B(x,r)$ of radius $\lambda r$ with $\mu(B')=0$.
\end{lem}

\begin{proof}  We may suppose that $x=0$ and $r=1$.  Furthermore, by increasing $\eps>0$ if necessary, we may assume that $\int_{B(0,1)}|\Tbar_{\mu}(1)(y)| dm_d(y)=\eps \omega_d$ (here $\omega_d$ is the volume of the $d$-dimensional unit ball).

Let $\gamma>0$.  To prove this lemma, we shall look to apply the alternative in Lemma \ref{colalt} to balls of radius $\gamma>0$.  To this end, we shall want to work with the function $\Tbar_{\mu, M\gamma}(1)$ for some $M>1$ to be chosen later.  First note that, whenever $\delta\in (0,1]$, inequality (\ref{localintest}) yields that
$$\int_{B(0,1)}|\langle T^{\delta}[\delta_x],1\rangle_{\mu}|dm_d(x)\leq C\delta^{d-s}\mu(B(0,2))\leq C\delta^{d-s},
$$
and so from identity (\ref{deltprodrep2}), we deduce that
\begin{equation}\begin{split}\label{localintest2}\eps \omega_d+ C\delta^{d-s}\geq\int_{B(0,1)}|\Tbar_{\mu, \delta}(1)(x)| dm_d(x)\geq\eps \omega_d- C\delta^{d-s}.
\end{split}\end{equation}

Consequently, as long as $M\gamma<\min(1,c_{13}\eps^{1/(d-s)})$ for a suitably chosen $c_{13}>0$, the second inequality in (\ref{localintest2}) yields that
\begin{equation}\label{Mgambig}\int_{B(0,1)} |\Tbar_{\mu, M\gamma}(1)(x)| dm_d(x)\geq \frac{\eps\omega_d}{2}.\end{equation}

Next, we shall derive a crude absolute value estimate for $\Tbar_{\mu, M\gamma}(1)$ in the ball $B(0,1)$.    To this end, note that from the first inequality in (\ref{localintest2}) with $\delta=1$, we see that there must be a point $x_0\in B(0,1)$ such that $|\Tbar_{\mu, 1}(1)(x_0)|\leq C\eps + C$.  But then  the H\"{o}lder continuity of $\Tbar_{\mu, 1}(1)$ (Lemma \ref{Lipest}) yields that $|\Tbar_{\mu, 1}(1)(x)|\leq C\eps+C$ for any $x\in B(0,2)$.  Now consider $F_{M\gamma,1}(x) = \Tbar_{\mu, M\gamma}(1)(x)-\Tbar_{\mu, 1}(1)(x)$.  Then (\ref{Fabsval}) yields that for any $x\in \R^d$,
\begin{equation}\begin{split}\nonumber|F_{M\gamma,1}(x)|&\leq \int_{B(x,1)}\frac{2}{(M\gamma+|x-y|)^s}d\mu(y)\\
&\leq C\frac{\mu(B(x,M\gamma))}{(\gamma M)^s}+C\int_{M\gamma}^1\frac{\mu(B(x,t))}{t^s}\frac{dt}{t}\\
&\leq C+C\log\bigl(\frac{1}{M\gamma}\bigl).
\end{split}\end{equation}
Bringing these observations together, we see that there is a constant $C_{14}>0$ such that
\begin{equation}\label{uniformtriv}|\Tbar_{\mu, M\gamma}(1)(x)|\leq C_{14}\bigl[ \eps+1+\log\bigl(\frac{1}{M\gamma}\bigl)\bigl] \text{ for every }x\in B(0,1).
\end{equation}

Now, take a maximal $\gamma$-separated set in $B(0,1)$.  Set $B_j = B(x_j, \gamma)$.  Then the balls $B_j$ form a cover of $B(0,1)$.  Furthermore, under our assumption that $M\gamma<1$, the enlarged balls $2MB_j = B(x_j, 2M\gamma)$ are contained in $B(0,3)$, and have covering number $C_{15}M^d$ (at most $C_{15}M^d$ balls $B(x_j, 2M\gamma)$ may contain any given point in $\mathbb{R}^d$).

With the aim of obtaining a contradiction, we suppose that $\mu(B_j)>0$ for all $j$.  Now introduce $\tau=\tau\bigl(\tfrac{\eps}{4}\bigl)>0$ and $M(\tfrac{\eps}{4})>0$ as in Lemma \ref{colalt}, and suppose that $M\geq M(\tfrac{\eps}{4})$.
For every fixed $j$, if $x\in B_j$ satisfies $$|\Tbar_{\mu, M\gamma}(1)(x)|>\frac{\eps}{4}+\frac{C_1}{M^{\alpha}},$$
then the H\"{o}lder continuity of $\Tbar_{\mu, M\gamma}(1)$ (see Lemma \ref{Lipest}) ensures that  $|\Tbar_{\mu, M\gamma}(1)(x_j)|>\tfrac{\eps}{4}.$  But then since we have assumed that $\mu(B_j)>0$, Lemma \ref{colalt} implies that $\mu(2MB_j)\geq \tau\gamma^s$.  However, note that
$$\sum_j\mu(2MB_j)\leq C_{15}M^d\mu(B(0,3))\leq CM^d.
$$
Thus, the balls $B_j$ that satisfy $\mu(2MB_j)\geq \tau\gamma^s$ can number at most $\tfrac{CM^d}{\tau\gamma^s},$ and so the union of these balls $B_j$ has volume (or $m_d$ measure) no greater than $\tfrac{\omega_dC_{16}M^d}{\tau}\gamma^{d-s}$.    Since $B(0,1)\subset \bigcup_j B_j$, our conclusion is that the set
$$ E=\Bigl\{x\in B(0,1): |\Tbar_{\mu, M\gamma}(1)(x)|>\frac{\eps}{4}+\frac{C_1}{M^{\alpha}}\Bigl\}
$$
has $m_d$ measure at most $\tfrac{\omega_dC_{16}M^d}{\tau}\gamma^{d-s}$.  Combined with (\ref{uniformtriv}), we get that
$$\int_{E}|\Tbar_{\mu, M\gamma}(1)(x)|dm_d\leq \frac{\omega_dC_{16}M^d}{\tau}\gamma^{d-s}C_{14}\bigl[ \eps+1+\log\bigl(\frac{1}{M\gamma}\bigl)\bigl].
$$
But of course,
$$\int_{B(0,1)\backslash E}|\Tbar_{\mu, M\gamma}(1)(x)|dm_d\leq \frac{\omega_d\eps}{4}+\frac{C_1\omega_d}{M^{\alpha}}.
$$
We therefore reach a contradiction with (\ref{Mgambig}) if $\frac{C_1}{M^{\alpha}}<\frac{\eps}{8}$,  and
\begin{equation}\label{paramterchoices}\frac{C_{16}M^d}{\tau}\gamma^{d-s}C_{14}\bigl[\eps+1+\log\bigl(\frac{1}{M\gamma}\bigl)\bigl]<\frac{\eps}{8}.
\end{equation}
In that case there must exist some $j$ with $\mu(B_j)=0$.  But since $x_j\in B(0,1)$, there is a ball of radius $\tfrac{\gamma}{2}$ contained in $B_j\cap B(0,1)$, and this ball is disjoint from $\supp(\mu)$.

It remains to make a choice of $M$ and then $\gamma$ to ensure that the conditions that have been placed upon these two parameters throughout the proof are consistent.  First let us fix $M>\max\bigl[M\bigl(\tfrac{\eps}{4}\bigl), \bigl(\tfrac{8C_1}{\eps}\bigl)^{1/\alpha}\bigl]$ (thereby fixing $M$ in terms of $\eps$).  When choosing $\gamma>0$, there are two conditions to take care of: $M\gamma<\min(1,c_{13}\eps^{1/(d-s)})$, and (\ref{paramterchoices}).  Since the left hand side of (\ref{paramterchoices}) tends to zero as $\gamma$ tends to zero, such a choice of $\gamma$ is clearly possible, and this completes the proof.\end{proof}

\section{A variant of Cotlar's inequality and Wiener's inversion lemma}

This section is concerned with proving two basic technical lemmas; a variant of Cotlar's inequality, and a variant of the Wiener lemma.  Both of these results will be used often in Parts II and III.  

\subsection{Cotlar's inequality}

\begin{lem}  There exists a constant $C>0$, depending on $s$, $d$, $\alpha$, and $\Lambda$, such that for any non-trivial $\Lambda$-nice reflectionless measure,
$$\sup_{\delta>0}|\Tbar_{\mu, \delta}(1)(x)|\leq C, \text{ for any }x\in \mathbb{R}^d.
$$
\end{lem}

Before we prove this lemma, let us note an immediate corollary of it.  If $f$ is $m_d$-measurable on $\R^d$, denote by $\|f\|_{L^{\infty}(m_d)}$ the essential supremum of $f$.  That is, the least $M>0$ for which $m_d(\{x\in \R^d: |f(x)|>M\})=0$.

\begin{cor}\label{reflinfbd}  If $\mu$ is a non-trivial $\Lambda$-nice reflectionless measure, then $\|\Tbar_{\mu}(1)\|_{L^{\infty}(m_d)}\leq C$.
\end{cor}

\begin{proof}[Proof of the Cotlar inequality]  The proof follows a standard path, based upon an idea of David and Mattila, see \cite{DM, NTrV}.   Let $\delta>0$, and set $B_j = B(x,5^j\delta)$.  Suppose that $\mu(B_{j+1})\geq 5^{s+1}\mu(B_{j})$ for all $j\in \mathbb{Z}_+$.  Insofar as $\mu$ is non-trivial, $\mu(B_{j'})>0$ for some $j'\in \mathbb{Z}_+$.  But then for $j>j'$, $\mu(B_j) \geq \mu(B_{j'})5^{(s+1)(j-j')}$, and so for sufficiently large $\mu(B_j) > \Lambda 5^{sj}\delta^s,$
which is a contradiction.  Thus, there is a least $j\in \mathbb{Z}_+$ with $\mu(B_{j+1}) < 5^{s+1}\mu(B_{j})$.  Set $r=5^j\delta$.  Then $\mu(B(x,r))>0$, and $\mu(B(x,5r))<5^{s+1}\mu(B(x,r))$.

 First note that
$$|\Tbar_{\mu, \delta}(1)(x) - \Tbar_{\mu, r}(1)(x)|\leq \int_{B(x,r)}|K_{\delta}(y-x)-K_{r}(y-x)|\,d\mu(y).
$$
The right hand side is trivially bounded by $2\int_{B(x,r)}\tfrac{\,d\mu(y)}{\max(\delta,|x-y|)^s}$.  But now note that this integral may in turn be estimated by a constant multiple of
$$\sum_{0\leq \ell\leq j}\frac{\mu(B(x, 5^{\ell}\delta))}{5^{\ell s}\delta^s}\leq \mu(B(x, 5^j\delta))\sum_{0\leq \ell\leq j}\frac{1}{5^{(s+1)(j-\ell)}5^{\ell s}\delta^s}.
$$
The sum on the right hand side has size at most $\Lambda 5^{js}\sum_{0\leq \ell\leq j}\tfrac{5^{\ell}}{5^{j(s+1)}}\leq C.$  From this we conclude that $|\Tbar_{\mu, \delta}(1)(x) - \Tbar_{\mu, r}(1)(x)|\leq C.$

Now choose a non-negative bump function $\psi\in \Lip_0(B(0,2))$ such $\psi \equiv 1$ on $B(0,1)$.  Set $\psi_{x,r} = \psi\bigl(\tfrac{\cdot -x}{r}\bigl)$.  It remains to estimate
$| \Tbar_{\mu, r}(1)(x)|$, which, according to the reflectionless property of $\mu$ is equal to $$\Bigl|\Tbar_{\mu,r}(1)(z) - \Bigl[\int_{\mathbb{R}^d}\psi_{x,r} \,d\mu\Bigl]^{-1}\langle \T1(\psi_{x,r}\mu),1\rangle_{\mu}\Bigl|.$$  But, appealing to (\ref{bardecentrep}) with $\nu=\psi_{x,r}\mu$ and $\varphi\equiv 1$, we see that this is in turn equal to
\begin{equation}\begin{split}\nonumber\Bigl| \Tbar_{\mu, r}(1)(x)-&\Bigl[\int_{\mathbb{R}^d}\psi_{x,r} \,d\mu\Bigl]^{-1}\int_{\R^d}\Tbar_{\mu,r}(1)\psi_{x,r}\,d\mu\\
&-\Bigl[\int_{\mathbb{R}^d}\psi_{x,r}\,d\mu\Bigl]^{-1}\langle T^r(\psi_{x,r}\mu),1\rangle_{\mu}\Bigl|.\end{split}\end{equation}
The H\"{o}lder continuity of $\Tbar_{\mu,r}(1)$ (see Lemma \ref{Lipest}) yields that
$$\Bigl|\Tbar_{\mu,r}(1)(x) - \Bigl[\int_{\mathbb{R}^d}\psi_{x,r} \,d\mu\Bigl]^{-1}\int_{\R^d}\Tbar_{\mu,r}(1)\psi_{x,r}\,d\mu\Bigl|\leq C.$$
On the other hand, applying Lemma \ref{lipgenest} with the choices $\delta=r$, $f=\psi_{x,r}$, and $\varphi\equiv 1$, yields
$$|\langle T^r(\psi_{x,r}\mu), 1\rangle_{\mu}| \leq Cr \|\psi_{x,r}\|_{\Lip}\mu(B(x,2r))\leq C\mu(B(x,2r)).
$$
Thus
$$\Bigl[\int_{\mathbb{R}^d}\psi_{x,r} \,d\mu\Bigl]^{-1} \bigl|\langle T^r(\psi_{x,r}\mu), 1 \rangle_{\mu}\bigl| \leq \frac{C\mu(B(x,2r))}{\mu(B(x,r))}\leq C,$$
where the doubling property was used in the final inequality.  Bringing these estimates together proves the lemma.
\end{proof}

\subsection{A Wiener Lemma}  Our next result is a variant of the Wiener inversion lemma.  Notice that a homogeneous CZ kernel $K$ can be written as $K(x) = \tfrac{\Omega(\tfrac{x}{|x|})}{|x|^s}$, where $\Omega :\mathbb{S}^{d-1}\to \mathbb{C}^{d'}$.  We shall assume (solely in this subsection) that $\Omega\in C^{\infty}(\mathbb{S}^{d-1})$.  Under this assumption, we have that
\begin{equation}\label{fourier}
\widehat{\frac{\Omega\bigl(\tfrac{\cdot}{|\cdot|}\bigl)}{|\cdot|^s}}(\xi) = \frac{m\bigl(\tfrac{\xi}{|\xi|}\bigl)}{|\xi|^{d-s}}, \text{ for any } \xi\neq 0,
\end{equation}
for a (vector valued) $m\in C^{\infty}(\mathbb{S}^{d-1})$ (see for example  Proposition 2.4.8 of Grafakos \cite{Gr}).  Furthermore, if $f\in \mathcal{S}(\mathbb{R}^d)$ (the Schwartz class) satisfies $f*K\in L^1(m_d)$, then $$\widehat{f*K}(\xi) = \hat{f}(\xi) \frac{m\bigl(\tfrac{\xi}{|\xi|}\bigl)}{|\xi|^{d-s}}.$$

\begin{lem}\label{wiener}  Suppose that $\mu$ is a $\Lambda$-nice measure, and $m\bigl(\xi)\neq 0$ for any $\xi\in \mathbb{S}^d$.  If, for some constant $\Gamma\in \mathbb{C}^{d'}$, $\Tbar_{\mu}(1)(x)= \Gamma$ for $m_d$-almost every $x\in \mathbb{R}^d$, then $\mu\equiv 0$. \end{lem}

This lemma can be proved by a slight modification of any of the proofs of Wiener's lemma based upon localization.  The proof that follows is based upon a paper of Korevaar \cite{Kor}.

\begin{proof} Choose $\eta \in \mathcal{S}(\mathbb{R}^d)$ satisfying $\widehat\eta \equiv 1$ on $B(0,1)$, $\widehat\eta \geq 0$ in $\mathbb{R}^d$, and $\widehat\eta \equiv 0$ outside $B(0,2)$.   Define $\eta_{\eps}$ by $\widehat\eta_{\eps}  = \widehat\eta\bigl(\tfrac{\cdot}{\eps}\bigl)$.  

Fix $\xi_0\neq 0$.  By assumption, there is a component $m_j$ of $m$ for which $|m_j(\tfrac{\xi_0}{|\xi_0|})|>0$.  Then there is a ball $B(\xi_0, t)$ with $0<t<\tfrac{|\xi_0|}{2}$, and $|m_j(\tfrac{\xi}{|\xi|})|\geq \tfrac{1}{2}|m_j(\tfrac{\xi_0}{|\xi_0|})|$ for any $\xi\in B(\xi_0, 2t)$.

Set $\psi = \mathcal{F}^{-1}\widehat\eta_t(\cdot - \xi_0)(x) = e^{2\pi i x\cdot \xi_0}t^n \eta(t\cdot x)$.  Note that $\psi$ is a Schwartz class function with $m_d$-mean zero (certainly $\widehat\eta_t(- \xi_0)=0$).  Now pick $\beta \in \mathbb{N}$ satisfying $s+2\beta>d$, and define $$G = [\Delta^{\beta} \psi]*K_j.$$

It is clear that $G$ is a smooth function.  We claim that it has the following two additional properties:

\begin{enumerate}
\item $|G(x)|\leq \frac{C}{(1+|x|)^{s+2\beta}}$ (so $G\in L^1(m_d)$), and
\item $G*\mu \equiv 0$ in $\mathbb{R}^d$.\end{enumerate}

 We shall first establish the decay estimate.  It is easy to see that $G$ is bounded, so it suffices to derive the claim for $|x|>1$.   For such an $x$, write
\begin{equation}\begin{split}\nonumber G(x) = &\int_{\mathbb{R}^d} [ \Delta^{\beta}_y \psi(x-y)] \widehat\eta(\tfrac{4(y-x)}{|x|}) K_j(y) dm_d(y)\\
& +   \int_{\mathbb{R}^d}[\Delta^{\beta}_y\psi(x-y)][1-\widehat\eta(\tfrac{4(y-x)}{|x|})] K_j(y) dm_d(y).
\end{split}\end{equation}
In order to estimate the first of the two terms on the right hand side, we integrate by parts (several times) to obtain
$$\int_{\mathbb{R}^d}   \psi(x-y) \Delta_y^{\beta}\bigl[\widehat\eta(\tfrac{4(y-x)}{|x|}) K_j(y)\bigl] dm_d(y).
$$
But $\bigl[\widehat\eta(\tfrac{4(\cdot-x)}{|x|}) K_j(y)\bigl]$ is supported in   $B(x, \tfrac{|x|}{2})$, and for $y\in(x, \tfrac{|x|}{2})$ we have that $|\Delta_y^{\beta}\bigl[\widehat\eta(\tfrac{4(y-x)}{|x|}) K_j(y)\bigl]|\leq \tfrac{C}{|x|^{s+2\beta}}$.  Thus
$$\Bigl|\int_{\mathbb{R}^d}   \psi(x-y) \Delta_y^{\beta}\bigl[\widehat\eta(\tfrac{4(y-x)}{|x|}) K_j(y)\bigl] dm_d(y)\Bigl|\leq  \tfrac{C}{|x|^{s+2\beta}}\|\psi\|_{L^1(m_d)}\leq  \tfrac{C}{|x|^{s+2\beta}}.$$
For the second term, merely note that $|\Delta^{\beta}\psi(x-y)|\leq \tfrac{C_n}{|x-y|^n}$ for any $n\in \mathbb{N}$.  Combined with the observation that $1-\widehat\eta(\tfrac{4(\cdot-x)}{|x|})$ is supported in $\mathbb{R}^d\backslash B(x, \tfrac{|x|}{4})$, this bound yields
\begin{equation}\begin{split}\nonumber\Bigl|\int_{\mathbb{R}^d}&[\Delta^{\beta}_y\psi(x-y)](1-\widehat\eta(\tfrac{4(y-x)}{|x|})) K_j(y) dm_d(y)\Bigl|\\
&\leq C_n\int_{\mathbb{R}^d\backslash B(x, \tfrac{|x|}{4})} \frac{1}{|x-y|^n}\frac{1}{|y|^s}dm_d(y)\leq \frac{C_n}{|x|^{n-(d-s)}}.\end{split}\end{equation}

 To see the second claim,  fix $x'$ with $\delta_{x'}\in \MBD(\mu)$ and $\Tbar_{\mu}(1)(x')= \Gamma.$  Recalling the formula (\ref{diffform}), we see that for $m_d$-almost every $x\in \mathbb{R}^d$,
 \begin{equation}\begin{split}\nonumber 0= \Tbar_{\mu}(1)(x) -\Tbar_{\mu}(1)(x')   =\int_{\mathbb{R}^d} [K(y-x)-K(y-x')]\,d\mu(y).\end{split}\end{equation}
Note that the decay estimate (1), along with the niceness of $\mu$, readily yields that there is a constant $C>0$ so that $|G|*\mu(x)\leq C$ for all $x\in \mathbb{R}^d$.  Thus
 \begin{equation}\begin{split}\nonumber G*\mu(x) &= \int_{\mathbb{R}^d}\int_{\mathbb{R}^d}\psi(z) K(x-y-z)dm_d(z)\,d\mu(y)\\
 & =  \int_{\mathbb{R}^d}\int_{\mathbb{R}^d}\psi(z) \bigl[K(x-y-z)-K(x'-y)\bigl]dm_d(z)\,d\mu(y)\\
 \end{split}\end{equation}
 where the $m_d$-mean zero property of $\psi$ has been used in order to freely subtract the $K(x'-y)$ term in the inner integral.  But then, as we shall prove momentarily, for every $x\in \mathbb{R}^d$,
 \begin{equation}\label{xprimedouble}\int_{\mathbb{R}^d}|\psi(z)|\int_{\mathbb{R}^d}\bigl|K(x-y-z)-K(x'-y)\bigl|\,d\mu(y)dm_d(z)<\infty,
 \end{equation}
  and so the Fubini theorem yields
 $$G*\mu(x)=\int_{\mathbb{R}^d}\psi(z)\int_{\mathbb{R}^d}\big[K(x-y-z)-K(x'-y)\bigl]\,d\mu(y)dm_d(z).
 $$
 But the inner integral equals zero for $m_d$-almost every $x\in \mathbb{R}^d$, which establishes the second claim (since $G*\mu$ is continuous).

  Let us now return to the claim (\ref{xprimedouble}), which is a pretty straightforward computation.  We shall split the inner integral into a number of pieces, regularly using the fact that $\psi$ lies in the Schwarz class.  First note that since $x'\in \MBD(\mu)$ is fixed, there is some constant $C>0$ such that $\int_{B(x',1)}|K(x'-y)|d\mu(y)\leq C$.  The standard tail estimate (\ref{standard}) also yields that,
 $$\int_{\mathbb{R}^d\backslash B(0, 2\max(|x'|, |x-z|))}|K(x-y-z)-K(x'-y)|d\mu(y)\leq C,
 $$
 where the constant depends on neither $x$ nor $z$.  Next, note that
 \begin{equation}\begin{split}\nonumber\int_{\mathbb{R}^d}|\psi(z)|\int_{B(x-z,1)}&|K(x-y-z)|d\mu(y)dm_d(z)\\
 &= \int_{\mathbb{R}^d}\int_{B(x-y,1)}|\psi(z)||K(x-y-z)|dm_d(z)d\mu(y).
 \end{split}\end{equation}
 But for $z\in B(x-y,1)$, we have $|\psi(z)|\leq \frac{C}{1+|x-y|^{s+1}}$, while it is easy to see that $\int_{B(x-y,1)}\tfrac{1}{|z-y-z|^s}dm_d(z)\leq C$, so
 $$ \int_{\mathbb{R}^d}\int_{B(x-y,1)}\!\frac{|\psi(z)|}{|x-y-z|^s}dm_d(z)d\mu(y)\leq \!C\!\!\int_{\mathbb{R}^d}\frac{1}{1+|x-y|^{s+1}}d\mu(y)\leq C,
 $$
 where the final inequality follows from the niceness of $\mu$.  It remains to estimate the sum of the two integrals
 \begin{equation}\begin{split}\nonumber\int\limits_{\mathbb{R}^d}&\int\limits_{B(0, 2\max(|x'|, |x-z|))\backslash B(x-z,1)}|K(x-z-y)|d\mu(y)|\psi(z)|dm_d(z)\\
 &+\int\limits_{\mathbb{R}^d}\int\limits_{B(0, 2\max(|x'|, |x-z|))\backslash B(x',1)}|K(x'-y)|d\mu(y)|\psi(z)|dm_d(z).
\end{split}\end{equation}
In the domains of integration of these integrals, the kernel $K$ is bounded by $1$ in absolutely value, so the niceness of $\mu$ yields $$2\int_{\mathbb{R}^d}|\psi(z)|\Lambda(2\max(|x'|, |x-z|))^sdm_d(z).
$$
However, $|\psi(z)|\leq \frac{C(x)}{{1+|x-z|}^{d+s+1}}$ for all $z\in \mathbb{R}^d$, from which it is easily seen that the previous integral is finite. Bringing these estimates together proves the claim (\ref{xprimedouble}).

Next $G\in L^1(m_d)$, and one readily calculates that $\widehat{G}(\xi) = b|\xi|^{2\beta}\widehat\eta_t(\xi-\xi_0)\widehat{K_j}(\xi)$ for some non-zero complex number $b$, and so $\widehat{G}(\xi)\neq 0$ in $B(\xi_0, t)$.   Let $\eps\in (0,\tfrac{t}{2})$, and consider the function $$\frac{\widehat\eta_{\eps}(\xi-\xi_0)}{\widehat{G}(\xi)}.$$  This function lies in the Schwartz class, and so it is the Fourier transform of a Schwartz class function $F$.

Now, since $|G|*\mu$ is a bounded function, we certainly have that $[|F|*(|G|*\mu)](x)<\infty$ for every $x\in \mathbb{R}^d$.  Thus $(F*G)*\mu = F*(G*\mu) \equiv 0 $ in $\mathbb{R}^d$.  But since $F*G = \mathcal{F}^{-1}\widehat\eta_{\eps}(\cdot-\xi_0)$, we obtain
$$\bigl[\mathcal{F}^{-1}\widehat\eta_{\eps}(\cdot-\xi_0)\bigl]*\mu=0.
$$
Taking the Fourier transform, we deduce that the tempered distribution $\hat{\mu}$ vanishes in the ball $B(\xi_0, \eps)$.  Since $\xi_0$ was taken to be any non-zero frequency, $\hat{\mu}$ is supported at the origin, and is therefore the Fourier transform of a polynomial.    But since $\mu$ is non-negative, so is the polynomial.  If the polynomial is non-zero then there is a constant $c>0$ such that for all sufficiently large $R$, $\mu(B(0,R))\geq cR^d$.  But $\mu(B(0,R))\leq \Lambda R^s$, and for large enough $R$ this yields a contradiction.
\end{proof}

\section{Weak convergence results}

In this section we establish some convergence results of the bilinear form defined in Section \ref{primer}.  Several of them have antecedents in Mattila and Verdera's paper \cite{MV} on the convergence of singular integrals.

We begin with a definition.  A sequence of measures $\mu_k$ is called \textit{uniformly diffuse} if, for each $R>0$ and $\eps>0$, there exists $r>0$ such that for all $k$,
\begin{equation}\label{closesmall}\iint\limits_{\substack{B(0,R)\times B(0,R)\\|x-y|<r}}\frac{\,d\mu_k(x)\,d\mu_k(y)}{|x-y|^{s-1}}\leq \eps.
\end{equation}

A sequence of measures $\mu_k$ is said to have \textit{uniformly restricted growth (at infinity)} if, for each $\eps>0$, there exists an $R\in(0,\infty)$ such that for all $k$,
\begin{equation}\label{tailR}\int_{\mathbb{R}^d\backslash \overline{B(0, R)}}\frac{1}{|x|^{s+\alpha}}\,d\mu_k(x)\leq \eps.
\end{equation}

It is easy to see that a measure $\mu$ is a diffuse if and only if for each $\eps>0$ and $R>0$, there exists $r>0$ such that (\ref{closesmall}) holds for $\mu$.  A measure $\mu$ has restricted growth at infinity if and only if  for every $\eps>0$, there exists $R>0$ such that (\ref{tailR}) holds for $\mu$.

We leave it to the reader to show that any sequence of $\Lambda$-nice measures $\mu_k$ is uniformly diffuse with uniformly restricted growth at infinity.  In future applications, it will be important to permit sequences of measures with more unusual growth conditions (see Section 5 of Part II), which accounts for the more general definition.

We recall that a sequence of measures $\mu_k$ is said to \textit{converge weakly} to a measure $\mu$ if $\lim_{k\to \infty}\int_{\mathbb{R}^d} \varphi(x)\,d\mu_k(x) = \int_{\mathbb{R}^d}\varphi(x)\,d\mu(x)$ for any $\varphi\in C_0(\mathbb{R}^d)$ (the space of compactly supported continuous functions).

\begin{lem} If $\mu_k$ is a weakly convergent sequence of uniformly diffuse measures (respectively measures with uniformly restricted growth), then the limit measure $\mu$ is diffuse (respectively has restricted growth at infinity).  
\end{lem}

In order to prove the lemma, we shall require the following useful fact:  \textit{Suppose that a sequence of measures $\mu_k$ converges weakly to $\mu$, then the sequence of product measures $\mu_k\times\mu_k$ converges weakly to $\mu\times\mu$, that is }
$$\lim_{k\to\infty}\int_{\mathbb{R}^d\times\mathbb{R}^d}\varphi(x,y)\,d\mu_k(x)\,d\mu_k(y) = \int_{\mathbb{R}^d\times\mathbb{R}^d}\varphi(x,y)\,d\mu(x)\,d\mu(y),
$$
for any $\varphi\in C_0(\mathbb{R}^d\times\mathbb{R}^d)$.  This is a standard exercise that can be proved by approximating a function in $  C_0(\mathbb{R}^d\times\mathbb{R}^d)$ by finite sums $\sum_j \psi_j\otimes\eta_j$ with $\psi_j,\eta_j\in C_0(\mathbb{R}^d)$ .

With this fact in hand, the proof of the lemma also becomes a simple exercise.  We shall prove the lemma in the case when the sequence of measures is uniformly diffuse, as the case of uniformly restricted growth is similar.  Fix $\eps>0$ and $R>0$.  Choose $r>0$ so that (\ref{closesmall}) holds for all $k$.  Notice that the set $U=\{(x,y)\in B(0,R)\times B(0,R) : |x-y|<r\}$ is open, and as such, the function
$$(x,y)\mapsto \frac{\chi_U(x,y)}{|x-y|^{s-1}}
$$
 is lower semi-continuous, and so is equal to the pointwise limit of a non-decreasing sequence of non-negative functions in $C_0(\mathbb{R}^d\times\mathbb{R}^d)$.  For each of these $C_0(\mathbb{R}^d\times\mathbb{R}^d)$ functions $\varphi$, we have $\int_{\mathbb{R}^d\times\mathbb{R}^d}\varphi(x,y)\,d\mu(x)\,d\mu(y)\leq \eps$ due to the weak convergence of $\mu_k\times\mu_k$ to $\mu\times\mu$.  But then the monotone convergence theorem yields that $$\iint\limits_{\substack{B(0,R)\times B(0,R)\\|x-y|<r}}\frac{\,d\mu(x)\,d\mu(y)}{|x-y|^{s-1}}\leq \eps$$ as required.

\begin{lem}\label{compsuppconv} If $\mu_k$ is a uniformly diffuse sequence of measures that converges weakly to a measure $\mu$ (and so $\mu$ is diffuse), then
$$\lim_{k\rightarrow \infty}\langle T(f\mu_k), \varphi \rangle_{\mu_k} = \langle T(f\mu), \varphi\rangle_{\mu},
$$
for any $f$ and $\varphi$ in $\Lip_0(\R^d)$.
\end{lem}

\begin{proof}  Fix $\delta>0$. The function $(x,y)\mapsto K_{\delta}(x-y) H_{f,\varphi}(x,y)$ lies in $C_0(\R^d\times \R^d)$, and so the weak convergence of the sequence of measures $\mu_k\times\mu_k$ to $\mu\times\mu$ ensures that
\begin{equation}\begin{split}\nonumber \lim_{k\to \infty}\iint_{\R^d\times\R^d}&K_{\delta}(x-y) H_{f,\varphi}(x,y)\,d\mu_k(y)\,d\mu_k(x) \\
&= \iint_{\R^d\times\R^d}K_{\delta}(x-y) H_{f,\varphi}(x,y)\,d\mu(y)\,d\mu(x).
\end{split}\end{equation}
In other words, $\lim_{k\to\infty}\langle T_{\delta}(f\mu_k), \varphi \rangle_{\mu_k} = \langle T_{\delta}(f\mu), \varphi\rangle_{\mu}.$

On the other hand, if $B(0,R)\supset \supp(f)\cup\supp(g)$, then
\begin{equation}\begin{split}|\nonumber\langle T_{\delta}(f\mu_k), \varphi \rangle_{\mu_k}&- \langle T(f\mu_k), \varphi \rangle_{\mu_k}|\\&=\Bigl|\int_{\mathbb{R}^d}H_{f,\varphi}(x,y)K^{\delta}(x-y)\,d\mu_k(x)\,d\mu_k(y)\Bigl| \\
&\leq C_{f,\varphi}\iint\limits_{\substack{B(0,R)\times B(0,R)\\|x-y|\leq\delta}}\frac{\,d\mu_k(y)\,d\mu_k(x)}{|x-y|^{s-1}}
\end{split}\end{equation}
(and the same inequalities hold true with $\mu$ replacing $\mu_k$ in every instance in the previous line).
Therefore, on account of the defining property of a uniformly diffuse sequence, we see that as $\delta\to 0$, $\langle T_{\delta}(f\mu_k), \varphi \rangle_{\mu_k}$ converges uniformly (in $k$) to $\langle T(f\mu_k), \varphi \rangle_{\mu_k}$, and also $\langle T_{\delta}(f\mu), \varphi \rangle_{\mu}$ converges to $\langle T(f\mu), \varphi \rangle_{\mu}$.  But when combined with the fact that $\lim_{k\to\infty}\langle T_{\delta}(f\mu_k), \varphi \rangle_{\mu_k} = \langle T_{\delta}(f\mu), \varphi\rangle_{\mu}$, this uniform convergence establishes that $\lim_{k\rightarrow \infty}\langle T(f\mu_k), \varphi \rangle_{\mu_k} = \langle T(f\mu), \varphi\rangle_{\mu}.$
\end{proof}

We shall need another simple fact about a uniformly diffuse sequence of measures, namely that the measure of a ball can be controlled uniformly.  More precisely, we have the following lemma.

\begin{lem}\label{uniformpleasgrowth}  Suppose that $\mu_k$ is a sequence of uniformly diffuse measures.   Then for every $R>0$, there is a constant $C(R)$ (that may depend on $R$) such that for every $k$, $\mu_k(B(0,R))\leq C(R)$.
\end{lem}

\begin{proof}Let $R>0$.  By hypothesis, there exists $r>0$ (which we may take to be smaller than $R$) such that
$$\iint\limits_{\substack{B(0,2R)\times B(0,2R)\\|x-y|<r}}\frac{\,d\mu_k(x)\,d\mu_k(y)}{|x-y|^{s-1}}\leq 1
$$
for all $k$.

Now, we may cover $B(0,R)$ with $C\bigl(\tfrac{R}{r}\bigl)^d$ balls $B(x_j, \tfrac{r}{2})$ with $x_j\in B(0,R)$, and $C>0$ depending on $d$. For each $j$,
$$\bigl\{(x,y)\in B(0,2R)\times B(0,2R) :|x-y|<r\bigl\}\supset B(x_j,\tfrac{r}{2})\times B(x_j,\tfrac{r}{2}).
$$
Consequently,
$$\frac{\mu(B(x_j, \tfrac{r}{2}))^2}{r^{s-1}}\leq\iint\limits_{B(x_j,\tfrac{r}{2})\times B(x_j,\tfrac{r}{2})}\frac{\,d\mu_k(x)\,d\mu_k(y)}{|x-y|^{s-1}}\leq 1.
$$
But then,
$$\mu(B(0,R))\leq C\Bigl(\frac{R}{r}\Bigl)^dr^{\tfrac{s-1}{2}},
$$
which yields the required estimate.
\end{proof}

The uniform growth at infinity condition plays a crucial role in the next convergence result.  In order to state it, we shall need to define a space of test functions.  For a measure $\mu$, and $R>0$, define
$$\Phi_R^{\mu} = \Bigl\{f\in \Lip_0(B(0,R)): \|f\|_{\Lip}<1\text{ and }\int_{\R^d}f\,d\mu=0\Bigl\},
$$
and
$$\Phi^{\mu} = \Bigl\{f\in \Lip_0(\R^d): \|f\|_{\Lip}<1\text{ and }\int_{\R^d}f\,d\mu=0\Bigl\}.
$$
We shall follow the notation that $\Phi_{R}^{\mu} = \Phi^{\mu}$ if $R=+\infty$.

\begin{lem}  Suppose that $\mu_k$ is a uniformly diffuse sequence of measures with uniformly restricted growth that converges weakly to a measure $\mu$ (and so $\mu$ is diffuse, and has restricted growth at infinity).  Suppose that $\gamma_k$ is a non-negative sequence converging to zero, and $R_k\in (0,+\infty]$ is a sequence converging to $R\in (0, +\infty]$.

If, for every $k$,
$$|\langle T(f\mu_k), 1 \rangle_{\mu_k}|\leq \gamma_k \text{ for every } f\in \Phi_{R_k}^{\mu_k},
$$
then
$$|\langle T(f\mu), 1\rangle_{\mu}|=0 \text{ for every } f\in \Phi_{R}^{\mu}.
$$
\end{lem}

\begin{proof} If $\mu(B(0,R))=0$, then there is nothing to prove, so let us assume that $\mu(B(0,R))>0$.  Fix $f\in \Phi_R^{\mu}$.  Then there exists $R'\in(0,R)$ such that $\mu(B(0,R'))>0$, $\supp(f)\subset B(0,R')$.  Clearly $R'\leq R_k$ for all sufficiently large $k$.  Choose a non-negative function $\rho\in \Lip_0(B(0,R'))$ with $\|\rho\|_{L^1(\mu)}=1$.  If $k$ is large enough, then $\|\rho_k\|_{L^1(\mu_k)}\geq \tfrac{1}{2}$.  For these $k$, set $f_k = f - \lambda_k \rho$, where $\lambda_k = \bigl(\int_{\R^d}\rho \,d\mu_k \bigl)^{-1}\int_{\R^d}f\,d\mu_k$.  Note that $\lambda_k\to 0$ as $k\to \infty$.  Consequently, for large enough $k$, $f_k\in \Phi_{R_k}^{\mu_k}$ and so $|\langle T(f_k\mu_k),1\rangle_{\mu_k}|\leq \gamma_k$.

For the remainder of the proof, $C>0$ denotes a constant that may depend on $f$, $\rho$, and $R'$, as well as $d$, $s$ and $\alpha$, and it may change from line to line.

Since $\lambda_k$ is a bounded sequence, $|f_k(x)|\leq C$ for every $x\in \mathbb{R}^d$.  Thus, for $x\not\in B(0, 2R')$, the $\mu_k$-mean zero property of $f_k$ yields that
\begin{equation}\begin{split}\nonumber|T&(f_k\mu_k)(x)| = \Bigl|\int_{\mathbb{R}^d}[K(x-y)-K(x)]f_k(y)\,d\mu_k(y)\Bigl|\\
&\leq \frac{C}{|x|^{s+\alpha}}\int_{\R^d}|f_k(y)||y|^{\alpha}\,d\mu_k(y) \leq C \sup_k \mu_k(B(0,R')) \frac{1}{|x|^{s+\alpha}}.
\end{split}\end{equation}
The same estimate also holds with $\mu_k$ and $f_k$ replaced by $\mu$ and $f$.  From Lemma \ref{uniformpleasgrowth}, we infer that, for $x\not\in B(0, 2R')$,
$$\sup_k |T(f_k\mu_k)(x)| \leq \frac{C}{|x|^{s+\alpha}}, \text{ and } |T(f\mu)(x)| \leq \frac{C}{|x|^{s+\alpha}}.
$$

Let  $\varphi\in \Lip_0(\mathbb{R}^d)$, $0\leq \varphi\leq1$.  Then the uniformly restricted growth of the sequence $\mu_k$ ensures that $\sup_k\bigl|\int_{\R^d}(1-\varphi(x))\frac{1}{|x|^{s+\alpha}}d(\mu_k+\mu)(x)\bigl|$ can be made as small as we want by choosing $\varphi$ to be identically equal to $1$ on a ball of sufficiently large radius, centred at the origin.

Let $\eps>0$.  Since $|\langle T(f_k\mu_k), 1-\varphi\rangle_{\mu_k}|\leq C\bigl|\int_{\R^d}(1-\varphi(x))\frac{1}{|x|^{s+\alpha}}d\mu_k(x)\bigl|$ if $\varphi\equiv 1$ on $B(0, 2R')$, the observations of the previous paragraph ensure that, with a judicious choice of $\varphi$,
$$|\langle T(f_k\mu_k), 1-\varphi\rangle_{\mu_k}|\leq \eps \text{ for all }k,\text{ and }|\langle T(f\mu), 1-\varphi\rangle_{\mu}|\leq \eps.
$$
However, by Lemma \ref{compsuppconv}, $\lim_{k\to \infty}\langle T(f\mu_k), \varphi\rangle_{\mu_k} = \langle T(f\mu), \varphi\rangle_{\mu}$, and also $\lim_{k\to \infty}|\langle T(\lambda_k\rho \mu_k), \varphi \rangle_{\mu_k}| =\bigl[ \lim_{k\to \infty}\lambda_k\bigl]  \cdot |\langle T(\rho\mu),\varphi\rangle_{\mu}|=0$, so $$\lim_{k\to \infty}\langle T(f_k\mu_k), \varphi\rangle_{\mu_k} = \langle T(f\mu), \varphi\rangle_{\mu}.$$  Bringing everything together, we see that \begin{equation}\begin{split}\nonumber |\langle T(f\mu), 1\rangle_{\mu}|&\leq \eps+\lim_{k\to \infty}|\langle T(f\mu_k, \varphi\rangle_{\mu_k}|\leq  2\eps+\limsup_{k\to \infty}|\langle T(f_k\mu_k, 1\rangle_{\mu_k}|\\
& \leq 2\eps+\limsup_{k\rightarrow\infty}\gamma_k = 2\eps,\end{split}\end{equation} from which the lemma follows.
\end{proof}

An immediate consequence of this lemma will prove to be a useful result in its own right.

\begin{cor}  Suppose that $\mu_k$ is a sequence of uniformly diffuse reflectionless measures with uniformly restricted growth that converges weakly to a measure $\mu$.  Then $\mu$ is a reflectionless measure.
\end{cor}

To prove the corollary, just pick $\gamma_k=0$, and $R_k=+\infty$ in the assumptions of the lemma prior to it.  Our final convergence lemma concerns the pointwise convergence of the potential $\Tbar_{\mu_k}(1)$ when $\mu_k$ is a sequence of measures.

\begin{lem}  Suppose that $\mu_k$ is a sequence of non-trivial uniformly diffuse reflectionless measures with uniformly restricted growth that converges weakly to a measure $\mu$ (and so $\mu$ is reflectionless).  Let $x\in \mathbb{R}^d$.  Assume that $\mu$ is non-trivial, and that there is a ball $B(x,\delta)$ that is disjoint from $\bigcup_{k\geq 1}\supp(\mu_k)$.  Then $$\lim_{k\to\infty}\Tbar_{\mu_k}(1)(x) = \Tbar_{\mu}(1)(x).$$
\end{lem}

\begin{proof}  First notice that $B(x,\delta)\cap\supp(\mu)=\varnothing$ (the weak limit is lower-semicontinuous).  We are required to show that $\lim_{k\to\infty}\Tbar_{\mu_k,\delta}(1)(x) = \Tbar_{\mu,\delta}(1)(x)$ (see Section 3.8).  Since $\mu$ is non-trivial, we can find a non-negative function $\eta\in \Lip_0(\R^d)$ with $\int_{\mathbb{R}^d}\eta \,d\mu=1$.  Set $\lambda_k = \bigl(\int_{\R^d}\eta \,d\mu_k\bigl)^{-1}$.  Then $\lambda_k\to 1$ as $k\to \infty$, and $\eta_k = \lambda_k \eta$ satisfies $\int_{\R^d}\eta_k \,d\mu_k =1$.  We shall henceforth suppose that $k$ is large enough to ensure that $\lambda_k\in (\tfrac{1}{2},2)$.

Recalling (\ref{Tbardeltdef}), write \begin{equation}\begin{split}\nonumber\Tbar_{\mu_k,\delta}(1)(x) = & \iint_{\R^d\times \R^d}[K_{\delta}(y-x) - K_{\delta}(y-z)]\eta_k(z)\,d\mu_k(y)\,d\mu_k(z)\\
& - \langle T^{\delta}(\eta_k\mu_k),1\rangle_{\mu_k}.\end{split}\end{equation}  Here the reflectionless property of $\mu_k$ was used insofar as the value of $\Tbar_{\mu_k,\delta}(1)(x)$ should be independent of the choice of the smooth probability measure $\nu_0$ (which we take to be $\eta_k\mu_k$).

Now choose a non-negative function $\varphi\in \Lip_0(\R^d)$, $0\leq\varphi\leq 1$ on $\mathbb{R}^d$, which is identically equal to $1$ on a ball $B(0,R)$, so that  $B(0, \tfrac{R}{2})\supset \supp(\eta)\cup\{x\}$.  Then
$$\Bigl|\iint_{\R^d\times \R^d}[K_{\delta}(y-x) - K_{\delta}(y-z)](1-\varphi(y))\eta_k(z)\,d\mu_k(y)\,d\mu_k(z)\Bigl|
$$
is bounded by a constant multiple of
\begin{equation}\begin{split}\nonumber\int_{\R^d} \frac{(1-\varphi(y))}{|y|^{s+\alpha}}\,d\mu_k(y)&\int_{\mathbb{R}^d}|x-z|^{\alpha}\eta_k(z)\,d\mu_k(z)\\&\leq C\int_{\R^d} \frac{(1-\varphi(y))}{|y|^{s+\alpha}}\,d\mu_k(y),
\end{split}\end{equation}
where $C>0$ may depend on $x$ and $\eta$ along with $s$, $d$, and $\alpha$.
But the right hand side may be made arbitrarily small for all $k$ (or with $\mu_k$ and $\eta_k$ replaced by $\mu$ and $\eta$), by choosing the radius $R>0$ appearing in the definition of $\varphi$ large enough (the uniformly restricted growth at infinity of the sequence $\mu_k$ is used here).  Thus, for every $\eps>0$, we may choose $R>0$ so large that for every $k$,
$$\Bigl|\iint_{\R^d\times \R^d}[K_{\delta}(y-x) - K_{\delta}(y-z)](1-\varphi(y))\eta_k(z)\,d\mu_k(y)\,d\mu_k(z)\Bigl|\leq\eps,
$$
and
$$\Bigl|\iint_{\R^d\times \R^d}[K_{\delta}(y-x) - K_{\delta}(y-z)](1-\varphi(y))\eta(z)\,d\mu(y)\,d\mu(z)\Bigl|\leq \eps.
$$

On the other hand, the function $$(y,z)\mapsto [K_{\delta}(y-x)-K_{\delta}(y-z)]\varphi(y)\eta(z)$$ lies in $C_0(\mathbb{R}^d\times\mathbb{R}^d)$, and so the weak convergence of $\mu_k$ to $\mu$ yields that the limit
$$\lim_{k\to \infty}\iint_{\R^d\times \R^d}[K_{\delta}(y-x) - K_{\delta}(y-z)]\varphi(y)\lambda_k\eta(z)\,d\mu_k(y)\,d\mu_k(z)
$$
exists, and is equal to
$$\iint_{\R^d\times \R^d}[K_{\delta}(y-x) - K_{\delta}(y-z)]\varphi(y)\eta(z)\,d\mu(y)\,d\mu(z).
$$
Since $T^{\delta}$ has its kernel supported in the ball $\overline{B(0,\delta)}$, $\langle T^{\delta}(\eta_k\mu_k),1\rangle_{\mu_k} = \lambda_k\langle T^{\delta}(\eta\mu_k), \psi\rangle_{\mu_k}$, where $\psi\in \Lip_0(\R^d)$ is identically equal to $1$ on the $\delta$-neighbourhood of $\eta$.  Since $\lambda_k\to 1$, Lemma \ref{compsuppconv}  yields that $\lim_{k\to\infty}\langle T^{\delta}(\eta_k\mu_k),1\rangle_{\mu_k}= \langle T^{\delta}(\eta\mu), \psi\rangle_{\mu}=\langle T^{\delta}(\eta\mu), 1\rangle_{\mu}$  (also recall here that $T^{\delta}= T-T_{\delta}$).

In conclusion, for each $\eps>0$, $\limsup_{k\to\infty}|\Tbar_{\mu_k,\delta}(1)(x) - \Tbar_{\mu,\delta}(1)(x)|\leq 2\eps,$
from  which the lemma follows.
\end{proof}

\end{document}